\documentclass[leqno,10pt]{amsart}
\usepackage[normal]{boilerplate}

\usepackage{tikz-cd}
\usetikzlibrary{matrix}
\usetikzlibrary{arrows}
\usetikzlibrary{decorations.pathmorphing}
\usetikzlibrary{trees}
\tikzset{
commutative diagrams/.cd, 
arrow style=tikz, 
diagrams={>=stealth}
}

\newcommand{\angs}[1]{\langle #1 \rangle}

\newcommand{\mb}{\mathbf}
\newcommand{\mc}{\mathcal}

\newcommand{\cop}{\copyright}

\newcommand{\D}{\Delta}

\newcommand{\fiss}{\textit{fiss}}

\newcommand{\inj}{\ \tikz[baseline]\draw[>=stealth, commutative diagrams/hookrightarrow](0,0.5ex)--(0.5,0.5ex);\ }
\newcommand{\iy}{\infty}
\newcommand{\lcm}{\text{lcm}}
\newcommand{\La}{\Lambda}
\newcommand{\Mack}{\categ{Mack}}
\newcommand{\Orb}{\categ{Orb}}
\newcommand{\os}{\overset}
\newcommand{\para}{\acwopencirclearrow}
\newcommand{\NNhat}{\widehat{\NN}}
\newcommand{\st}{\textit{st}}
\newcommand{\Symm}{\textup{Symm}}

\newcommand{\twa}{\widetilde{\mathcal{O}}}
\newcommand{\ul}{\underline}

\newcommand{\X}{\times}

\DeclareMathOperator{\THH}{TH}

\title{Cyclonic spectra, cyclotomic spectra, and a conjecture of {K}aledin}

\author{Clark Barwick}
\address{Massachusetts Institute of Technology, Department of Mathematics, Bldg. 2, 77 Massachusetts Ave., Cambridge, MA 02139-4307}
\email{clarkbar@gmail.com}

\author{Saul Glasman}
\address{School of Mathematics, Institute for Advanced Study, Fuld Hall, 1 Einstein Dr., Princeton NJ 08540}
\email{sglasman@math.ias.edu}

\begin{document}

\begin{abstract}
With an explicit, algebraic indexing $(2,1)$-category, we develop an efficient homotopy theory of \emph{cyclonic objects}: circle-equivariant objects relative to the family of finite subgroups. We construct an $\infty$-category of cyclotomic spectra as the homotopy fixed points of an action of the multiplicative monoid of the natural numbers on the category of cyclonic spectra. Finally, we elucidate and prove a conjecture of Kaledin on cyclotomic complexes.
\end{abstract}

\maketitle

\setcounter{tocdepth}{1}
\tableofcontents


\setcounter{section}{-1}

\section{Summary} We construct the homotopy theory of cyclotomic spectra by means of a completely algebraic approach -- without reference to equivariant orthogonal spectra, equivariant $S$-modules, functors with smash product, etc., and we employ it, in this paper and its sequels, to realize a program suggested by Kaledin \cite{MR2827805,MR2918295,MR3137194}.

The construction proceeds in four steps:
\begin{enumerate}[(1)]
\item We begin with the category of $\QQ/\ZZ$-sets with finitely many orbits, all of which are of the form
\[\angs{m}\coloneq\QQ\Big/\frac{1}{m}\ZZ.\]
\item Then we sprinkle in additional $2$-isomorphisms -- one for every rational number that rotates one $\QQ/\ZZ$-equivariant map into another. In effect, these $2$-isomorphisms introduce a free homotopy between the identity and a generator in each cyclic group
\[\angs{m}_N\coloneq\frac{1}{N}\ZZ\Big/\frac{1}{m}\ZZ.\]
The result is a $2$-category (that is, a category enriched in groupoids), $\FF_{\cop}$, whose objects we call \emph{cyclonic sets}.
\item Next, using the technology introduced by the first author in \cite{mack1}, we define \emph{cyclonic spectra} as \emph{spectral Mackey functors} -- that is, as additive functors from a $2$-category $A^{\eff}(\FF_{\cop})$ of spans of objects of $\FF_{\cop}$ to the $\infty$-category $\Sp$ of spectra. In effect, then, a cyclonic spectrum consists of the following data:
\begin{itemize}
\item for every positive integer $m$, a spectrum $X\angs{m}$,
\item for every positive integer $m$, an action $\gamma_m$ of $\QQ/\frac{1}{m}\ZZ$ on $X\angs{m}$,
\item for every pair of positive integers $m$ and $n$ such that $m$ divides $n$, two maps
\[\phi_{m|n}^{\star}\colon\fromto{X\angs{n}}{X\angs{m}}\text{\quad and\quad}\phi_{m|n,\star}\colon\fromto{X\angs{m}}{X\angs{n}},\]
\item for every positive integer $m$, every $s,t\in\QQ/\frac{1}{m}\ZZ$, and every rational number $r$ such that $r\equiv(s-t)\mskip-2mu\mod\frac{1}{n}\ZZ$, a homotopy $\rho_r\colon\gamma_m(s)\simeq\gamma_m(t)$,
\end{itemize}
all subject to a long list of coherence conditions determined by the $2$-category structure on $A^{\eff}(\FF_{\cop})$; in particular, one has the Mackey condition: if $m$ and $m'$ both divide $n$, then there is a distinguished homotopy
\[\phi_{m'|n}^{\star}\phi_{m|n,\star}\simeq\sum_{x\in C_n/C_{\lcm(m,m')}}\phi_{\gcd(m,m')|m',\star}\gamma_{\gcd(m,m')}(x)\phi_{\gcd(m,m')|m}^{\star}\colon\fromto{X\angs{m}}{X\angs{m'}}.\]
\item The multiplicative monoid $\NN$ of positive integers acts via $n\colon\goesto{\angs{m}}{\angs{mn}}$ on the $2$-category $\FF_{\cop}$, which in turn induces an action of $\NN$ on $\Sp_{\cop}$ via ``geometric fixed point'' functors. We then define the homotopy theory $\Sp_{\Phi}$ of cyclotomic spectra as the homotopy fixed points for this action:
\[\Sp_{\Phi}\coloneq(\Sp_{\cop})^{h\NN}.\]
\end{enumerate}
There is, of course, a $p$-typical version of this story for any prime $p$, in which all the positive integers that appear above are powers of $p$.

In either case, the homotopy theory of cyclotomic spectra enjoys a simple universal property. We also compare our homotopy theory with the one constructed by Blumberg--Mandell \cite{BM}, thereby proving that their homotopy theory enjoys the very same universal property.

The value of our fully algebraic approach is that nothing special is used in this story about the homotopy theory of spectra, apart from the fact that it's additive and that it admits suitable colimits. We are therefore entitled to replace $\Sp$ with any homotopy theory $\AA$ with these properties in the recipe above and to form the homotopy theories $\AA_{\cop}$ of \emph{cyclonic Mackey functors valued in $\AA$} and $\AA_{\Phi}$ of \emph{cyclotomic objects of $\AA$}.
\begin{itemize}
\item If $\AA$ is taken to be the ordinary category $\Ab$ of abelian groups, we find that what we get is the category $\Ab_{\Phi}$ of ordinary Mackey functors indexed on the divisibility poset, along with additional restriction functors; this provides the structure naturally seen on the big ring of Witt vectors.
\item Similarly, if $\AA$ is taken to be the homotopy theory $\DD(\ZZ)$ of chain complexes of abelian groups, then $\DD(\ZZ)_{\Phi}$ can be compared with Kaledin's cyclotomic complexes \cite{MR3137194}.
\end{itemize}

Finally, we prove a conjecture of Kaledin \cite[(0.1)]{MR3137194}: we show that for any commutative ring $R$, there is a pullback square of ``noncommutative brave new schemes''
\begin{equation*}
\begin{tikzpicture}[baseline]
\matrix(m)[matrix of math nodes,
row sep=4ex, column sep=4ex,
text height=1.5ex, text depth=0.25ex]
{\Spec\DD(R)_{\Psi} & \Spec\Sp_{\Psi} \\
\Spec\DD(R) & \Spec\Sp \\ };
\path[>=stealth,->,font=\scriptsize]
(m-1-1) edge node[above]{} (m-1-2)
edge node[left]{} (m-2-1)
(m-1-2) edge node[right]{} (m-2-2)
(m-2-1) edge node[below]{} (m-2-2);
\end{tikzpicture}
\end{equation*}
inducing, after applying ``quasicoherent sheaves,'' the square of left adjoints
\begin{equation}\label{eqn:Kaledinsquare}
\begin{tikzpicture}[baseline]
\matrix(m)[matrix of math nodes,
row sep=4ex, column sep=4ex,
text height=1.5ex, text depth=0.25ex]
{\DD(R)_{\Phi} & \Sp_{\Phi} \\
\DD(R) & \Sp. \\ };
\path[>=stealth,<-,font=\scriptsize]
(m-1-1) edge node[above]{} (m-1-2)
edge node[left]{} (m-2-1)
(m-1-2) edge node[right]{} (m-2-2)
(m-2-1) edge node[below]{} (m-2-2);
\end{tikzpicture}
\end{equation}

This means that there is a  square of stable, presentable $\infty$-categories and left adjoints
\begin{equation*}
\begin{tikzpicture}[baseline]
\matrix(m)[matrix of math nodes,
row sep=4ex, column sep=4ex,
text height=1.5ex, text depth=0.25ex]
{\DD(R)_{\Psi} & \Sp_{\Psi} \\
\DD(R) & \Sp, \\ };
\path[>=stealth,<-,font=\scriptsize]
(m-1-1) edge node[above]{} (m-1-2)
edge node[left]{} (m-2-1)
(m-1-2) edge node[right]{} (m-2-2)
(m-2-1) edge node[below]{} (m-2-2);
\end{tikzpicture}
\end{equation*}
that exhibits $\DD(R)_{\Psi}$ as the tensor product $\DD(R)\otimes\Sp_{\Psi}$, which we regard as a tensor product of \emph{noncommutative derived rings}, such that the induced diagram
\begin{equation*}
\begin{tikzpicture}[baseline]
\matrix(m)[matrix of math nodes,
row sep=4ex, column sep=4ex,
text height=1.5ex, text depth=0.25ex]
{\Fun^L(\DD(R)_{\Psi},\Sp) & \Fun^L(\Sp_{\Psi},\Sp) \\
\Fun^L(\DD(R),\Sp) & \Fun^L(\Sp,\Sp), \\ };
\path[>=stealth,<-,font=\scriptsize]
(m-1-1) edge node[above]{} (m-1-2)
edge node[left]{} (m-2-1)
(m-1-2) edge node[right]{} (m-2-2)
(m-2-1) edge node[below]{} (m-2-2);
\end{tikzpicture}
\end{equation*}
on \emph{modules} coincides with the square \eqref{eqn:Kaledinsquare} above. (The bottom arrows are well-known to agree already \cite[Th. 4.8.4.1]{HA}.) This we prove in the final section.

\subsection*{Future work} In a sequel to this paper, we show that the topological Hochschild homology ($\THH$) of any Waldhausen $\infty$-category admits the natural structure of a cyclotomic spectrum. Again our approach is fully algebraic, and it produces a functor
\[\THH\colon\fromto{\DD_{\fiss}(\Wald_{\infty})}{\Sp_{\Phi}},\]
which we may think of this as the \emph{noncommutative syntomic realization} functor. We shall address the connection to the ``classical'' syntomic realization in the fullness of time.

Furthermore, the flexible algebraic set-up we have provided has allowed us to develop an important variation on this story. This is the theory of what we call \emph{multicyclonic} and \emph{multicyclotomic} spectra. In this variant, the unit corepresents the higher versions of topological cyclic homology introduced by Brun, Carlsson, Douglas, and Dundas \cite{MR2729005,MR2737802}. We study this in detail in a further sequel.

\subsection*{Origins} In 2009, the first author gave a talk at MIT in which he described an approach to the de Rham--Witt complex based on spectral Mackey functors for $\QQ/\ZZ$. Unfortunately, the construction he offered was inelegant, because the residual actions of the circle group were not, as it were, baked into the pie. It wasn't until early 2015 that the second author showed how to modify the orbit category of $\QQ/\ZZ$ by adding $2$-isomorphisms in just the right way. This lead rapidly to the present article.

\subsection*{Acknowledgments} The authors thank the participants in the Bourbon Seminar -- particularly Emanuele Dotto, Marc Hoyois, Denis Nardin, and Jay Shah -- for many, many hours of productive and insightful discussion. Additionally, much of the writing on this paper was completed while the first author was enjoying the excellent working conditions at the University of Glasgow; he thanks the University and the Fulbright Foundation.

This material is partly based on work done while the second author was supported by the National Science Foundation under agreement No. DMS-1128155. Any opinions, findings and conclusions, or recommendations expressed in this material are those of the authors and do not necessarily reflect the views of the National Science Foundation.


\section{Cyclonic orbits and cyclonic sets}

\begin{nul} 
By a \emph{strict $2$-category}, we will mean an ordinary category enriched in groupoids. As usual, we will call the objects of the morphism groupoids the \emph{$1$-morphisms} and the isomorphisms of the morphism groupoids the \emph{$2$-isomorphisms}.
\end{nul}

\begin{cnstr}\label{cnstr:NCDelta} From a $2$-category $\CC$, one can obtain an $\iy$-category in two steps: first, we form the nerve of each morphism groupoid to obtain a fibrant simplicial category $\CC_\D$; next we form the simplicial nerve $N(\CC_\D)$ as in \cite[1.1.5.5]{HTT}.

Just as for $1$-categories, this procedure may be conducted in one explicit step as follows: an $n$-simplex of $N(\CC_\D)$ is the following data:
\begin{itemize}
\item a list $(X_0, X_1, \dots, X_n)$ of objects of $\CC$;
\item for each $i \leq j$, a vertex $\phi_{ij}$ of the groupoid $\Mor(X_i, X_j)$, which is the identity if $i = j$;
\item and for each $i \leq j \leq k$, a $2$-isomorphism $\alpha_{ijk}$ from $\phi_{ik}$ to $\phi_{jk}\circ\phi_{ij}$ in $\Mor(X_i, X_k)$, which is the identity if $i = j = k$, such that for each $i \leq j \leq k \leq l$, the identity
\[\alpha_{ijk}\circ\alpha_{ikl} = \alpha_{jkl}\circ\alpha_{ijl}\]
holds in $\Mor(X_i, X_l)$.
\end{itemize}

Since ultimately our work in this paper will take place in the context of $\iy$-categories, if $\CC$ is any $2$-category (even a $1$-category!), we shall abuse notation slightly and write $\CC$ for the corresponding $\iy$-category $N(\CC_{\D})$.
\end{cnstr}

\begin{ntn} We write $\NN$ for the set of positive integers, ordered by divisibility. We write $\NN_0$ for the set of nonnegative integers.
\end{ntn}

\begin{rec}\label{ntn:setofpossibleNs} Recall (\cite[\S 1.3]{MR1867431}) that a \emph{supernatural number} (in the sense of Steinitz) is a formal expression
\[N=\prod_{p}p^{v_{p}(N)},\]
where the product runs over all prime numbers $p$, and each $v_{p}(N)\in\NN_0\cup\{\infty\}$. We write $\NNhat$ for the set of supernatural numbers.

We identify $\NN$ with the subset of $\NNhat$ consisting of those elements $N\in\NNhat$ for which the $v_{p}(N)$ are all finite, and we call such an $N$ \emph{finite}. At the other extreme, one writes $\infty$ (or perhaps $0$) for the maximal element with $v_p(\infty)=\infty$ for any $p$.

Supernatural numbers form a commutative monoid under the obvious notion of multiplication (with $v_p(MN)=v_p(M)+v_p(N)$). We regard $\NNhat$ as a poset under the corresponding notion of divisibility, with respect to which it is a complete lattice.

For any element $N\in\NNhat$, let us denote by
\[\frac{1}{N}\ZZ\subseteq\QQ\]
the subgroup of all rational numbers $r$ such that no prime $p$ divides the denominator of $r$ more than $v_{p}(N)$ times. The assignment $\goesto{N}{\frac{1}{N}\ZZ}$ provides an ordered bijection between $\NNhat$ and the set of additive subgroups of $\QQ$ that contain $\ZZ$, ordered by inclusion. It follows that $\NNhat\cong\widehat{\ZZ}/\widehat{\ZZ}^{\times}$ (which justifies the notation to some extent).
\end{rec}

\begin{ntn} For any element $N\in\NNhat$, let us write $\NN_N$ for the set of finite divisors of $N$, ordered by divisibility. (This is elsewhere called the \emph{nest associated with} $N$.)

For any element $N\in\NNhat$, let us write $C_N$ for the torsion abelian group $\frac{1}{N}\ZZ/\ZZ$ (so that $N$ is the supernatural order of $C_N$ \cite[Df. 1.1.6(iii)]{MR2392026}). For any $m\in\NN_N$, let us write $\angs{m}_N$ for the $C_N$-orbit
\[\frac{1}{N}\ZZ\Big/\frac{1}{m}\ZZ.\]
Note that any transitive $C_N$-set with finite stabilizers is of this form.

We will also contemplate the action of $\frac{1}{N}\ZZ$ on $\angs{m}_N$ via the quotient map $\fromto{\frac{1}{N}\ZZ}{C_N}$. Since we're in an abelian context, we will always write this action additively.
\end{ntn}

\begin{wrn} Because of the potential for confusion, we wish to emphasize that the $C_N$-orbit $\angs{m}_N$ does \emph{not} have cardinality $m$. It turns out that we want to keep track of the cardinality of the \emph{stabilizers} instead, so $\angs{m}_N$ is the (unique up to isomorphism) $C_N$-orbit whose stabilizers have cardinality $m$.
\end{wrn}

\begin{exm} When $N$ is finite, $C_N$ is of course the usual cyclic group of order $N$, and any $C_N$-orbit is of the form $\angs{m}_N$ for some $m\in\NN_N$.
\end{exm}

\noindent We are particularly interested in two infinite examples: $N=\infty$ and $N=p^{\infty}$.

\begin{exm} When $N=\infty$, we have $C_N=\QQ/\ZZ$. We will often just write $\angs{m}$ for the $\QQ/\ZZ$-orbit $\angs{m}\coloneq\QQ/\frac{1}{m}\ZZ$.
\end{exm}

\begin{exm} If $p$ is a prime number and if $N=p^{\infty}$, then $C_N$ is the $p$-quasicyclic group $\QQ_p/\ZZ_p$. We are contemplating $\QQ_p/\ZZ_p$-orbits of the form
\[[k]_p=\langle p^k\rangle_{p^{\infty}}\coloneq\QQ_p\Big/\frac{1}{p^k}\ZZ_p\]
for some $k\in\NN_0$.
\end{exm}

\begin{dfn} Suppose $N\in\NNhat$. We write $\OO_{C_N}$ for the $1$-category whose objects are $C_N$-orbits of the form $\angs{m}_N$ and whose morphisms are $C_N$-equivariant maps. We call $\OO_{C_N}$ the \emph{degree $N$ orbit category of $C_N$}.

If $\angs{m}_N$ and $\angs{n}_N$ are two $C_N$-orbits and if $u$ and $v$ are two $C_N$-equivariant maps
\[\fromto{\angs{m}_N}{\angs{n}_N},\]
then an \emph{intertwiner from $u$ to $v$} is a number $r\in\frac{1}{N}\ZZ$ such that for any $z\in\angs{m}_N$, one has
\[v(z)\equiv r+u(z)\mskip-3mu\mod\frac{1}{n}\ZZ.\]
We now write $\OO_{\cop_N}$ for the following strict $2$-category.
\begin{itemize}
\item An object of $\OO_{\cop_N}$ is a $C_N$-orbit.
\item If $\angs{m}_N$ and $\angs{n}_N$ are $C_N$-orbits, then a $1$-morphism $\fromto{\angs{m}_N}{\angs{n}_N}$ is a $C_N$-equivariant map.
\item A $2$-isomorphism $\fromto{u}{v}$ is an intertwiner from $u$ to $v$.
\end{itemize}
Vertical composition in the morphism groupoid
\[\Mor_{\OO_{\cop_N}}(\angs{m}_N,\angs{n}_N)\]
is given by addition of integers, and for three objects $\angs{l}_N$, $\angs{m}_N$, and $\angs{n}_N$, the composition functor
\[\fromto{\Mor_{\OO_{\cop_N}}(\angs{l}_N,\angs{m}_N)\times\Mor_{\OO_{\cop_N}}(\angs{m}_N,\angs{n}_N)}{\Mor_{\OO_{\cop_N}}(\angs{l}_N,\angs{n}_N)}\]
is given by composition of functions, and the horizontal composition of $2$-morphisms is also given by addition of integers. We call $\OO_{\cop_N}$ the \emph{degree $N$ cyclonic orbit $2$-category}.
\end{dfn}

\begin{exm} The degree $1$ cyclonic orbit $2$-category is a $2$-groupoid with a unique object $\angs{1}_1$ and a unique $1$-morphism in which the set of $2$-morphisms is $\ZZ$. In other words, $\OO_{\cop_1}$ is precisely $BB\ZZ$.
\end{exm}

\begin{nul} A $2$-simplex of $\OO_{\cop_N}$ (viewed as an $\infty$-category) is a diagram
\begin{equation*}
\begin{tikzpicture}[baseline]
\matrix(m)[matrix of math nodes,
row sep=2ex, column sep=2ex,
text height=1.5ex, text depth=0.25ex]
{ &\angs{m}_N&  \\
&r& \\[-1.5ex]
\angs{l}_N && \angs{n}_N \\ };
\path[>=stealth,->,font=\scriptsize]
(m-3-1) edge[inner sep=0.75pt] node[above left]{$\phi$} (m-1-2)
edge node[below]{$\chi$} (m-3-3)
(m-1-2) edge[inner sep=0.75pt] node[above right]{$\psi$} (m-3-3);
\end{tikzpicture}
\end{equation*}
in which $\phi$, $\psi$, and $\chi$ are equivariant maps, and $r$ is an intertwiner from $\chi$ to $\psi\circ\phi$.
\end{nul}

\begin{nul} Observe that the morphism groupoid between any two objects of $\OO_{\cop_N}$ is a connected groupoid isomorphic to the group $\ZZ$, and so all of the mapping spaces in $\OO_{\cop_N}$ are circles. In particular, we have the object $(C_N, g)$, where $C_N$ is the cyclic group and $g$ is translation by the generator $1/N$, and $\End_{\OO_{\cop_N}} (C_N, g)$ is equivalent to the circle group $T$. It follows that the functor 
\[h_N \colon \OO_{\cop_N} \to \Top\]
 corepresented by $(C_N, g)$ lifts to a functor 
\[h_N^{T} \colon \OO_{\cop_N} \to \Top^{T},\]
where $\Top^{T}$ is the $\iy$-category of Top complexes with $T$-action and equivariant maps. In fact, $h_N^{T}$ is fully faithful and its essential image is spanned by the $T$-spaces $T/C_k$ as $k$ ranges over divisors of $N$.
\end{nul}

\begin{wrn} Since $\OO_{C_N}$ has the same objects and $1$-morphisms as $\OO_{\cop_N}$, we may consider the inclusion
\[\psi\colon\fromto{\OO_{C_N}}{\OO_{\cop_N}},\]
but please observe that this is \emph{not} the inclusion of a subcategory in the sense of \cite[\S 1.2.11]{HA}. Indeed, it follows from our computation of the mapping spaces that the functor
\[s\colon\fromto{\OO_{\cop_N}}{\NN_N}\]
defined by the assignment $\goesto{\angs{m}_N}{m}$ (i.e., the order of a stabilizer) exhibits $\NN_N$ as the homotopy category $h\OO_{\cop_N}$, and this does not contain $\OO_{C_N}$ as a subcategory.
\end{wrn}

\begin{ntn} When $N=\infty$, we drop the subscript and write $\OO_{\cop}$ for $\OO_{\cop_{\infty}}$.
\end{ntn}

\begin{dfn} For any $N\in\NNhat$, an \emph{$N$-cyclonic space} is a left fibration
\[X\to\OO_{\cop_N}^{\op}.\]
More generally, an \emph{$N$-cyclonic $\iy$-category} is a cocartesian fibration
\[X\to\OO_{\cop_N}^{\op}.\]
\end{dfn}

\begin{ntn} We shall write $\Top_{\cop_N}$ for the simplicial nerve of the full simplicial subcategory of $s\Set_{/\OO_{\cop_N}^{\op}}^{f}$ of simplicial sets over $\OO_{\cop_N}$ that is spanned by the left fibrations. Similarly, we write $\Cat_{\infty,\cop_N}$ for the simplicial nerve of the full simplicial subcategory of $s\Set_{/\OO_{\cop_N}^{\op}}^{+,f}$ of marked simplicial sets over $\OO_{\cop_N}$ that is spanned by the cocartesian fibrations (with exactly the cocartesian edges marked).
\end{ntn}

\begin{nul} Of course by straightening/unstraightening, one has equivalences of $\infty$-cate\-gories
\[\Top_{\cop_N}\simeq\Fun(\OO_{\cop_N}^{\op},\Top)\text{\quad and\quad}\Cat_{\infty,\cop_N}\simeq\Fun(\OO_{\cop_N}^{\op},\Cat_{\infty}).\]
Consequently, an $N$-cyclonic space (respectively, $\infty$-category) is essentially the data of a space (resp., $\infty$-category) $X$ together with a $T$-action and \emph{genuine} (not homotopy) fixed point spaces (resp., $\infty$-categories) $X^{C_M}$ for all the finite subgroups $C_M\subset T$ such that $M$ divides $N$. 
\end{nul}

\begin{dfn} For any $N\in\NNhat$, the \emph{$\iy$-category $\FF_{\cop_N}$ of finite cyclonic sets of degree $N$} is the closure of $\OO_{\cop_N}$ under formal finite coproducts. That is, $\FF_{\cop_N}$ is the smallest full subcategory of the $\iy$-category $\Top_{\cop_N}$ of $N$-cyclonic spaces that contains the essential image of the Yoneda embedding and is closed under finite coproducts.
\end{dfn}

\begin{nul} More explicitly, $\FF_{\cop_N}$ can be identified with the following $2$-category.
\begin{itemize}
\item The objects are $C_N$-sets whose stabilizers are all finite.
\item A $1$-morphism $\fromto{X}{Y}$ between two such $C_N$-sets is a $C_N$-equivariant map.
\item For any two $1$-morphisms $u,v\colon\fromto{X}{Y}$, a $2$-isomorphism is an \emph{intertwiner} from $u$ to $v$, by which we mean a tuple $(r_U)_{U\in\Orb(X)}$ of intertwiners -- indexed by the set $\Orb(X)$ of orbits of $X$ -- from $u|_U$ to $v|_U$.
\end{itemize}
\end{nul}

\begin{nul} Clearly the functor $\psi$ above extends to a functor
\[\psi\colon\fromto{\FF_{C_N}}{\FF_{\cop_N}},\]
where $\FF_{C_N}$ is the $1$-category of finite $C_N$-sets. This is also \emph{not} the inclusion of a subcategory.
\end{nul}

\begin{nul} The category $\OO_G$ of $G$-orbits for a finite group $G$ has the magical property that it acquires pullbacks upon adjoining formal finite coproducts: the pullback of a diagram of $G$-orbits exists as a finite $G$-set, which is after all just a finite coproduct of $G$-orbits in a canonical fashion. This is tautologically equivalent to the following property: for every functor
\[W\colon\Lambda^2_2 \to \OO_G\]
the category $(\OO_G)_{/W}$ has finitely many connected components, and each connected component admits a final object. We say that $\OO_G$ \emph{admits multipullbacks}.

We will apply the powerful Mackey functor machinery of \cite{mack1} to our categories of cyclonic sets, and for this, it will be necessary to show that the cyclonic orbit categories also admit multipullbacks.  (In the language of \cite{BDGNS}, they are \emph{orbital} $\iy$-categories.)
\end{nul}

\begin{dfn}
Let $\CC$ be a $2$-category, and let $X$ be an object of $\CC$. We define a category enriched in groupoids $\CC_{/X}$, the homotopically correct overcategory of $X$, as follows:
\begin{itemize}
\item An object of $\CC_{/X}$ is a morphism $f\colon Y \to X$ in $\CC$.
\item A morphism in $\CC_{/X}$ from $f\colon Y \to X$ to $g\colon Z \to X$ is a morphism $h\colon Y \to Z$ in $\CC$ together with a 2-morphism $k\colon f \to g \circ h$ in $\CC(Y, X)$.
\item An edge between $(h_1, k_1)$ and $(h_2, k_2)$ is a 2-morphism $q\colon h_1 \to h_2$ in $\CC(Y, Z)$ such that the identity $(g \circ q) \circ k_1 = k_2$ holds in $\CC(X, Z)$.
\end{itemize}
It's easy to see that this construction is compatible with the formation of overcategories in the $\iy$-categorical sense; that is, one has a natural isomorphism $N((\CC_{/X})_{\D})\cong (N(\CC_{\D}))_{/X}$.
\end{dfn}

\begin{nul}\label{disovecat}
We now investigate the overcategories of objects of $\OO_{\cop_N}$; in particular, we'll see that they're simply $1$-categories. Let $\angs{n}_N$ be an object of $\OO_{\cop_N}$ and let
\begin{equation*}
\begin{tikzpicture}[baseline]
\matrix(m)[matrix of math nodes,
row sep=2ex, column sep=2ex,
text height=1.5ex, text depth=0.25ex]
{ &\angs{m}_N&  \\
&s& \\[-1.5ex]
\angs{l}_N && \angs{n}_N \\ };
\path[>=stealth,->,font=\scriptsize]
(m-3-1) edge[inner sep=0.75pt] node[above left]{$\phi$} (m-1-2)
edge node[below]{$\alpha$} (m-3-3)
(m-1-2) edge[inner sep=0.75pt] node[above right]{$\beta$} (m-3-3);
\end{tikzpicture}
\begin{tikzpicture}[baseline]
\matrix(m)[matrix of math nodes,
row sep=2ex, column sep=2ex,
text height=1.5ex, text depth=0.25ex]
{ &\angs{m}_N&  \\
&t& \\[-1.5ex]
\angs{l}_N && \angs{n}_N \\ };
\path[>=stealth,->,font=\scriptsize]
(m-3-1) edge[inner sep=0.75pt] node[above left]{$\psi$} (m-1-2)
edge node[below]{$\alpha$} (m-3-3)
(m-1-2) edge[inner sep=0.75pt] node[above right]{$\beta$} (m-3-3);
\end{tikzpicture}
\end{equation*}
be a parallel pair of morphisms in $(\OO_{\cop_N})_{/\angs{n}_N}$. Then unwinding the definitions shows that an isomorphism from the left diagram to the right diagram is given by $r\in\frac{1}{N}\ZZ$ with denominator dividing $N$ such that for any $x\in\angs{l}_N$, one has $\psi(x) = r + \phi(x)$ and $t=r+s$. Such an $r$ is unique if it exists. In particular, $(\OO_{\cop_N})_{/\angs{n}_N}$ is equivalent to a $1$-category. Moreover, every morphism in $(\OO_{\cop_N})_{/\angs{n}_N}$ is uniquely isomorphic to a triangle whose filler is the identity $2$-morphism, which shows that the natural functor
\[\psi_S\colon(\OO_{C_N})_{/\angs{n}_N} \to (\OO_{\cop_N})_{/\angs{n}_N}\]
is an equivalence.
\end{nul}

\begin{prp}\label{cycmulpul}
For each $N\in\NNhat$, the $2$-category $\OO_{\cop_N}$ admits multipullbacks.
\begin{proof}
Let $\mb{C}$ be any $\iy$-category and let $W \colon \La_2^2 \to \mb{C}$ be a functor. Then we have a homotopy pullback diagram of categories
\[\begin{tikzcd}
\mb{C}_{/W} \ar{r} \ar{d} & \mb{C}_{/W(0)} \ar{d} \\
\mb{C}_{/W(1)} \ar{r} & \mb{C}_{/W(2)}.
\end{tikzcd}\]
Note that any functor $W\colon\La_2^2\to\OO_{\cop_N}$ can be lifted in a unique fashion to a functor $W'\colon\La_2^2\to\OO_{C_N}$. Together with \ref{disovecat}, this implies that 
\[\psi_W\colon(\OO_{C_N})_{/W'}\to(\OO_{\cop_N})_{/W}\]
is an equivalence of categories, and so we deduce the existence of multipullbacks in $\OO_{\cop_N}$ from the existence of multipullbacks in $\OO_{C_N}$.
\end{proof}
\end{prp}

\begin{cor}
For each $N\in\NNhat$, the $2$-category $\FF_{\cop_N}$ is a disjunctive $\iy$-category.
\begin{proof}
Obviously $\FF_{\cop_N}$ has finite coproducts, and it has pullbacks by \ref{cycmulpul}. It follows from the fact that the coproducts and pullbacks in $\FF_{\cop_N}$ are the coproducts and pullbacks in the $\iy$-category $\categ{P}(\OO_{\cop_N})$ that pullbacks distribute over coproducts and coproducts are disjoint and universal. 
\end{proof}
\end{cor}

\begin{dfn}\label{dfn:IMN} Suppose $M,N\in\NNhat$ such that $M$ divides $N$. There is an obvious inclusion $\into{\NN_M}{\NN_N}$, and our presentations for the groups $C_M$ and $C_N$ give a canonical inclusion $C_M\inj C_N$ determined by the condition that $1/M$ is carried to $1/M$. We therefore define a functor $I_M^N\colon\OO_{\cop_M}\to\OO_{\cop_N}$ as follows.
\begin{itemize}
\item On objects and $1$-morphisms, $I_M^N$ is the induction functor
\[-\times_{C_M}C_N\colon\goesto{\angs{n}_M}{\angs{n}_N}.\]
\item On $2$-morphisms, which are after all just suitable rational numbers, $I_M^N$ is the inclusion.
\end{itemize}
\end{dfn}

\begin{nul} The following diagram commutes:
\[\begin{tikzcd}[row sep = large]
\OO_{\cop_M} \ar{r}{s} \ar{d}[left]{I_M^N} & \NN_M \ar[hookrightarrow]{d}{} \\
\OO_{\cop_N} \ar{r}[below]{s} & \NN_N,
\end{tikzcd}\]
the functor $I_{M}^{N}$ now formally extends to a functor
\[I_{M}^{N}\colon\FF_{\cop_M}\to\FF_{\cop_N},\]
which makes the diagram
\[\begin{tikzcd}[row sep = large]
\FF_{C_M} \ar{r}{\psi} \ar{d}[left]{\text{Ind}_{C_M}^{C_N}} & \FF_{\cop_M} \ar{d}{I_M^N} \\
\FF_{C_N} \ar{r}[below]{\psi} & \FF_{\cop_N}
\end{tikzcd}\]
commute up to homotopy. We claim that $I_M^N$ preserves pullbacks; indeed, every pullback square in $\FF_{\cop_M}$ is, up to equivalence, the image of a pullback square in the category $\FF_{C_M}$ of finite $C_N$-sets, and both $\text{Ind}_{C_M}^{C_N}$ and $\psi$ preserve pullbacks.
\end{nul}

\begin{nul} Suppose $N\in\NNhat$. Since $\frac{1}{N}\ZZ$ can be written as the filtered union of the subgroups $\frac{1}{M}\ZZ$ over the set of finite divisors $M$ of $N$, it follows that we can write the $2$-categories above as filtered colimits:
\[\OO_{\cop_N}\simeq\underset{M\in\NN_{N}}{\colim}\ \OO_{\cop_M}\textrm{\quad and\quad}\FF_{\cop_N}\simeq\underset{M\in\NN_{N}}{\colim}\ \FF_{\cop_M},\]
and the $\infty$-categories above as cofiltered limits:
\[\Top_{\cop_N}\simeq\lim_{M\in\NN_{N}^{\op}}\Top_{\cop_M}\textrm{\quad and\quad}\Cat_{\infty,\cop_N}\simeq\lim_{M\in\NN_{N}^{\op}}\Cat_{\infty,\cop_M}.\]

That is, the assignments $\goesto{N}{\OO_{\cop_N}}$ and $\goesto{N}{\FF_{\cop_N}}$ define functors
\[\OO_{\cop_{\ast}}\colon\fromto{\NNhat}{\Cat_{\infty}}\text{\quad and\quad}\FF_{\cop_{\ast}}\colon\fromto{\NNhat}{\Cat_{\infty}}\]
that are left Kan extended from $\NN$, and the functors $\goesto{N}{\Top_{\cop_N}}$ and $\goesto{N}{\Cat_{\infty,\cop_N}}$ define functors
\[\Top_{\cop_{\ast}}\colon\fromto{\NNhat^{\op}}{\Cat_{\infty}}\text{\quad and\quad}\Cat_{\infty,\cop_{\ast}}\colon\fromto{\NNhat^{\op}}{\Cat_{\infty}}\]
that are right Kan extended from $\NN$.
\end{nul}


\section{Cyclonic spectra} 

\begin{nul} Suppose $N\in\NNhat$. Then since $\FF_{\cop_N}$ is a disjunctive $\infty$-category in the sense of \cite{mack1}, one can form its \emph{effective Burnside $\infty$-category} $A^{\eff}(\FF_{\cop_N})$. Now $A^{\eff}$ increases categorical level by one, so these $\infty$-categories are a priori $3$-categories, but a closer look reveals that in fact they are again simply $2$-categories.

Indeed, $A^{\eff}(\FF_{\cop_N})$ can be described in the following manner. An object is a finite cyclonic set of degree $N$, hence a disjoint union of cyclonic orbits. Between cyclonic orbits, a $1$-morphism is a sum of \emph{span diagrams}
\begin{equation*}
\begin{tikzpicture}[baseline]
\matrix(m)[matrix of math nodes, 
row sep={7ex,between origins}, column sep={7ex,between origins}, 
text height=1.5ex, text depth=0.25ex] 
{&\angs{l}_N&\\ 
\angs{m}_N&&\angs{n}_N.\\}; 
\path[>=stealth,->,font=\scriptsize] 
(m-1-2) edge (m-2-1) 
edge (m-2-3); 
\end{tikzpicture}
\end{equation*}
A $2$-morphism between two such diagrams is a diagram
\begin{equation*}
\begin{tikzpicture}[baseline]
\matrix(m)[matrix of math nodes, 
row sep={7ex,between origins}, column sep={7ex,between origins}, 
text height=1.5ex, text depth=0.25ex] 
{&&[-3ex]\angs{l}_N&[-3ex]&\\ 
\angs{m}_N&s&&t&\angs{n}_N\\
&&\angs{l'}_N&&\\}; 
\path[>=stealth,->,font=\scriptsize] 
(m-1-3) edge (m-2-1)
edge[inner sep=0.75pt] node[right]{$\phi$} (m-3-3)
edge (m-2-5) 
(m-3-3) edge (m-2-1) 
edge (m-2-5); 
\end{tikzpicture}
\end{equation*}
of $\FF_{\cop_N}$, where $\phi$ is an isomorphism, and $s$ and $t$ are intertwiners. Now for any diagrams
\begin{equation*}
\begin{tikzpicture}[baseline]
\matrix(m)[matrix of math nodes, 
row sep={7ex,between origins}, column sep={7ex,between origins}, 
text height=1.5ex, text depth=0.25ex] 
{&&[-3ex]\angs{l}_N&[-3ex]&\\ 
\angs{m}_N&s&&t&\angs{n}_N\\
&&\angs{l'}_N&&\\}; 
\path[>=stealth,->,font=\scriptsize] 
(m-1-3) edge (m-2-1)
edge[inner sep=0.75pt] node[right]{$\phi$} (m-3-3)
edge (m-2-5) 
(m-3-3) edge (m-2-1) 
edge (m-2-5); 
\end{tikzpicture}
\text{\quad and\quad}
\begin{tikzpicture}[baseline]
\matrix(m)[matrix of math nodes, 
row sep={7ex,between origins}, column sep={7ex,between origins}, 
text height=1.5ex, text depth=0.25ex] 
{&&[-3ex]\angs{l}_N&[-3ex]&\\ 
\angs{m}_N&s'&&t'&\angs{n}_N,\\
&&\angs{l'}_N&&\\}; 
\path[>=stealth,->,font=\scriptsize] 
(m-1-3) edge (m-2-1)
edge[inner sep=0.75pt] node[right]{$\phi'$} (m-3-3)
edge (m-2-5) 
(m-3-3) edge (m-2-1) 
edge (m-2-5); 
\end{tikzpicture}
\end{equation*}
a $3$-morphism from the $2$-morphism on the left to the $2$-morphism on the right is an intertwiner $r$ from $\phi$ to $\phi'$ such that
\[s'=s+r\text{\quad and\quad}t'=t+r,\]
which is clearly unique if it exists.

Passing to a skeleton, one sees that the mapping space in $A^{\eff}(\FF_{\cop_N})$ between cyclonic orbits $\angs{m}_N$ and $\angs{n}_N$ can be described as the following groupoid: an object is a formal sum of divisors $l$ of $\gcd(m,n)$, and a $1$-isomorphism is a family of $1$-automorphisms of such divisors $l$, which are equivalence classes of pairs $(s,t)$ consisting of an element $s\in\frac{1}{m}\ZZ$ and an element $t\in\frac{1}{n}\ZZ$, where two pairs $(s,t)$ and $(s',t')$ are equivalent if and only if $s'-s=t'-t$. In other words, $\Map_{A^{\eff}(\FF_{\cop_N})}(\angs{m}_N,\angs{n}_N)$ is the free $E_{\infty}$ space generated by the disjoint union
\[\coprod_{l\in\NN_{\gcd(m,n)}}B\left(\left(\frac{1}{m}\ZZ\oplus\frac{1}{n}\ZZ\right)\Big/\frac{1}{\gcd(m,n)}\ZZ\right)\simeq\coprod_{l\in\NN_{\gcd(m,n)}}B\left(\frac{1}{\lcm(m,n)}\ZZ\right).\]

Now we are interested in \emph{Mackey functors for $\FF_{\cop_N}$.} These are simply direct-sum-preserving functors from $A^{\eff}(\FF_{\cop_N})$ to a selected additive $\infty$-category.
\end{nul}

\begin{dfn} Suppose $N\in\NNhat$. We define the $\infty$-category of \emph{$N$-cyclonic spectra} as the $\infty$-category of Mackey functors
\[\Sp_{\cop_N}\coloneq\Mack(\FF_{\cop_N};\Sp).\]

More generally, for any additive $\infty$-category $\AA$, we define the $\infty$-category of \emph{$N$-cyclonic Mackey functors in $A$} as the $\infty$-category
\[\AA_{\cop_N}\coloneq\Mack(\FF_{\cop_N};\AA).\]
\end{dfn}

Before we dive into a study of cyclonic spectra proper, let us first dip a toe in the water by contemplating what happens when $\AA$ is merely a $1$-category.
\begin{exm}\label{exm:cyclonic1cat} Suppose $\AA$ an ordinary additive category, and suppose $N\in\NNhat$. Then the category $\AA_{\cop_N}$ is the ordinary category of direct-sum-preserving functors
\[\fromto{hA^{\eff}(\FF_{\cop_N})}{\AA}\]
Importantly, the result is actually \emph{less} structure than a Mackey functor for $C_N$, because the actions are all trivialized by the $2$-isomorphisms in $\FF_{\cop_N}$. Consequently, $\AA_{\cop_N}$ is nothing more than the category of Mackey functors for the divisibility poset $\Phi_N$ is the following category. That is, an object is a collection of the following data:
\begin{itemize}
\item for any finite divisor $m$ of $N$, an object $X\angs{m}$ of $\AA$; and
\item for any finite divisors $m$ and $n$ of $N$ such that $m$ divides $n$, two morphisms
\[\phi_{m|n,\star}\colon\fromto{X\angs{m}}{X\angs{n}}\text{\quad and\quad}\phi_{m|n}^{\star}\colon\fromto{X\angs{n}}{X\angs{m}};\]
\end{itemize}
all subject to the following conditions:
\begin{itemize}
\item for any finite divisors $u$, $v$, and $w$ of $N$ such that $u$ divides $v$ and $v$ divides $w$, one has
\[\phi_{v|w,\star}\phi_{u|v,\star}=\phi_{u|w,\star}\text{\quad and\quad}\phi_{u|v}^{\star}\phi_{v|w}^{\star}=\phi_{u|w}^{\star};\]
and 
\item for any finite divisor $m$ of $N$ and any divisors $k$ and $l$ of $m$, one has
\[\phi_{l|m}^{\star}\phi_{k|m,\star}=\frac{m}{\lcm(k,l)}\phi_{\gcd(k,l)|l,\star}\phi_{\gcd(k,l)|k}^{\star}\colon\fromto{X\angs{k}}{X\angs{l}}.\]
\end{itemize}
A morphism $\psi\colon\fromto{X}{Y}$ from one such object to another is a collection of maps
\[\psi_m\colon\fromto{X\angs{m}}{Y\angs{m}}\]
for each finite divisor $m$ of $N$ such that
\[\psi_n\phi_{m|n,\star}=\phi_{m|n,\star}\psi_m\text{\quad and\quad}\psi_m\phi_{m|n}^{\star}=\phi_{m|n}^{\star}\psi_n.\]
\end{exm}

\begin{subexm} Perhaps the most well-known example of this sort of structure is provided by the various rings of Witt vectors. To describe, let us begin with the Hopf algebra $\Symm(N)$. 

In particular, for any commutative ring $R$ and any natural number $n$, one may contemplate the ring $\WW_{\mskip-2mu\angs{m}}(R)$ of Witt vectors relative to the truncation set of positive divisors of $m$. We recall that this is the set $R^{\NN_m}$, equipped with the addition and multiplication maps that are uniquely determined by the condition that the \emph{ghost component} map
\[\goesto{(w_k)_{k\in\NN_m}}{(z_k)_{k\in\NN_m}=\left(\sum_{l\in\NN_k}lw_l^{k/l}\right)_{k\in\NN_m}}\]
is a ring map, functorially in $R$.

When $m$ divides $n$, there are two maps that appear between these objects: the \emph{Frobenius}
\[F_{m|n}\colon\fromto{\WW_{\mskip-2mu\angs{n}}(R)}{\WW_{\mskip-2mu\angs{m}}(R)},\]
given on ghost components by $\goesto{(z_l)_{l\in\NN_n}}{(z_{kn/m})_{k\in\NN_m}}$ and the \emph{Verschiebung}
\[V_{m|n}\colon\fromto{\WW_{\mskip-2mu\angs{m}}(R)}{\WW_{\mskip-2mu\angs{n}}(R)},\]
given on Witt components by $\goesto{(w_k)_{k\in\NN_m}}{(\overline{w}_l)_{l\in\NN_n}}$, where $\overline{w}_l=w_l$ if $l$ divides $m$ and $\overline{w}_l=0$ otherwise. It is a simple matter to see that these maps satisfy the conditions listed above. Consequently, these two maps endow the assignment $\goesto{\angs{m}}{\WW_{\mskip-2mu\angs{m}}(R)}$ with the structure of an object of $\Mod(R)_{\cop}$. Note, however, that the restriction maps that are usually taken as part of the structure on this assignment are not emergent from the cyclonic structure alone. To get these, one has to exhibit a cyclotomic structure. We will turn to this matter in the next section.

We can now restrict this Mackey functor to $N=p^{\infty}$ to get the $p$-typical Witt vectors everyone knows and loves.
\end{subexm}

\begin{exm}\label{exm:representableMack} For any $N\in\NNhat$ and any finite divisor $m$ of $N$, one has the Mackey functor
\[S^{\langle m\rangle}\colon\fromto{A^{\eff}(\FF_{\cop_N})}{\Sp}\]
corepresented by $\langle m\rangle_N$. By \cite[Cor. 9.2.1, Th. 13.12]{mack1} this carries any object $\langle n\rangle_N$ to the direct sum $K$-theory of the $2$-category
\[(\FF_{\cop_N})_{/\langle m\rangle}\times_{\FF_{\cop_N}}(\FF_{\cop_N})_{/\langle n\rangle}\simeq(\FF_{C_N})_{/\langle m\rangle}\times_{\FF_{\cop_N}}(\FF_{C_N})_{/\langle n\rangle}\]
(i.e., the $K$-theory in which the ingressives are declared to be the summand inclusions). By Barratt--Priddy--Quillen, this is the suspension spectrum
\[\bigvee_{l\in\NN_{\gcd(m,n)}}\Sigma^{\infty}_+B\left(\frac{1}{\lcm(m,n)}\ZZ\right).\]
\end{exm}


\begin{exm} Suppose $R$ a commutative ring and $N\in\NNhat$. We consider the (unbounded) derived $\infty$-category $\DD(R)$, which is also equivalent to $\Mod(HR)$. As with the $\infty$-category of cyclonic spectra, the $\infty$-category $\DD(R)_{\cop_N}$ of cyclonic complexes of $R$-modules admits a compact generator, which is given by
\[\bigoplus_{m\in\NN_N}M\angs{m}_N,\]
where $M\angs{m}_N$ is the cyclotomic complex corresponding to $HR\wedge\SS^{\angs{m}}$.

The Schwede--Shipley theorem now applies to provide a tilting equivalence
\[\DD(R)_{\cop_N}\simeq\Mod(E),\]
where $E$ is the endomorphism algebra
\[E\coloneq\End\left(\bigoplus_{m\in\NN_N}M\angs{m}_N\right).\]
Accordingly, we can identify $\DD(R)_{\cop_N}$ as the dg nerve \cite[Cnstr. 1.3.1.6]{HA} of a dg category of dg modules over the following dga: it is a quotient of the free dga
\[R\left[\left\{\alpha_{b,k,a},\epsilon_{b,k,a}\ |\ a,b\in\NN_N,\ k\in\NN_{\gcd(a,b)}\right\}\right],\]
where $\deg(\alpha_{b,k,a})=0$ and $\deg(\epsilon_{b,k,a})=1$, subject only to the following relations:
\begin{itemize}
\item for any $a,b,c,d\in\NN_N$, any $k\in\NN_{\gcd(a,b)}$, and any $l\in\NN_{\gcd(c,d)}$, one has
\[\epsilon_{d, l, c} \epsilon_{b, k, a} = 0;\]
\item for any $a,b,c,d\in\NN_N$, $k\in\NN_{\gcd(a,b)}$, and $l\in\NN_{\gcd(c,d)}$, if $b \neq c$, one has
\[\alpha_{d, l, c} \alpha_{b, k, a} = \epsilon_{d, l, c} \alpha_{b, k, a} = \alpha_{d, l, c} \epsilon_{b, k, a} = 0;\]
\item for any $a,b,c\in\NN_N$, $k\in\NN_{\gcd(a,b)}$, and $l\in\NN_{\gcd(b,c)}$, one has
\begin{eqnarray}
\alpha_{c, l, b} \alpha_{b, k, a} &=& \frac{b}{\lcm(k, l)} \alpha_{c, \gcd(l, k), a};\nonumber\\
\epsilon_{c, l, b} \alpha_{b, k, a} &=& \frac b k \epsilon_{c, \gcd(l, k), a};\nonumber\\
\alpha_{c, l, b} \epsilon_{b, k, a} &=& \frac b l \epsilon_{c, \gcd(l, k), a}.\nonumber
\end{eqnarray}
\end{itemize}
\end{exm}

\begin{nul} The functor $\fromto{\NNhat}{\Cat_2}$ that carries any $N$ to $\FF_{\cop_N}$ is left Kan extended from $\NN$, and all the overcategories $\NN_N$ are filtered. Consequently, the induced functor $\fromto{\NNhat}{\Cat_2}$ given by $\goesto{N}{A^{\eff}(\FF_{\cop_N})}$ is left Kan extended from $\NN$, and in turn the assignment $\goesto{N}{\Sp_{\cop_N}}$ is a functor
\[\fromto{\NNhat^{\op}}{\Cat_{\infty}}\]
that is right Kan extended from $\NN$.

In particular, we have identifications
\[\Sp_{\cop}\simeq\lim_{N\in\NN^{\op}}\Sp_{\cop_N}\text{\quad and\quad}\Sp_{\cop_{p^{\infty}}}\simeq\lim_{l\in\NN_0^{\op}}\Sp_{\cop_{p^l}}.\]
\end{nul}

We now wish to relate the $\infty$-categories $\Sp_{\cop_N}$ of cyclonic spectra to better-known models of $T$-equivariant spectra. We will need to be precise about which ones.

\begin{cnstr} Write $T$ for the circle group, and suppose $\mathcal{F}$ a family of closed subgroups of $T$ that is closed under passage to subgroups. Let $E\mathcal{F}$ be a universal space for $\mathcal{F}$.
\begin{itemize}
\item We call a map $\fromto{X}{Y}$ of orthogonal $T$-spectra an \emph{$\mathcal{F}$-equivalence} if, for any $H\in\mathcal{F}$, the induced map
\[\fromto{\pi_{\star}^HX}{\pi_{\star}^HY}\]
is an isomorphism of graded abelian groups – or, equivalently, if the map
\[\fromto{E\mathcal{F}^{\mskip1mu+}\wedge X}{E\mathcal{F}^{\mskip1mu+}\wedge Y}\]
is a $\ul{\pi}_{\star}$-isomorphism.
\item We will say that an orthogonal $T$-spectrum $X$ is \emph{$\mathcal{F}$-acyclic} if $E\mathcal{F}^{\mskip1mu+}\wedge X\simeq 0$.
\item Lastly, we will say that an orthogonal $T$-spectrum $Y$ is \emph{$\mathcal{F}$-local} if every map from an $\mathcal{F}$-acyclic is nullhomotopic.
\end{itemize}
Denote by $\categ{OrthogSp}_T$ the underlying $\infty$-category of the stable model category of orthogonal $G$-spectra. The full subcategory of $\mathcal{F}$-local spectra is then a localization of $\categ{OrthogSp}_T$, which we shall denote $\categ{OrthoSp}_{T,\mathcal{F}}$. Further, if $\mathcal{F}^{\mskip1mu\prime}\subset\mathcal{F}$ is a subfamily, then of course the full subcategory
\[\categ{OrthoSp}_{T,\mathcal{F}^{\mskip1mu\prime}}\subset\categ{OrthoSp}_{T,\mathcal{F}}\]
is also a localization.

For any $N\in\NNhat$, we have the family $\mathcal{F}_N$ of finite subgroups of $T$ whose order divides $N$. Let us simply write $\categ{OrthogSp}_{T,N}$ for $\categ{OrthogSp}_{T,\mathcal{F}_N}$.

Since these localizations can all be modeled as left Bousfield localizations of model categories, it follows that the assignment $\goesto{N}{\categ{OrthoSp}_{T,N}}$ is a diagram $\fromto{\NNhat^{\op}}{\Cat_{\infty}}$. One sees immediately that these are each right Kan extended from $\NN^{\op}$.
\end{cnstr}


\begin{thm}\label{thm:cyclonesareTspectra} There is an equivalence of $\infty$-categories
\[\equivto{\Sp_{\cop_{\ast}}}{\categ{OrthoSp}_{\TT,\ast}}\]
of $\Fun(\NNhat^{\op},\Cat_{\infty})$.
\begin{proof} Since both source and target are right Kan extended from $\NN^{\op}$, it suffices to construct this equivalence for the restrictions of these diagrams to $\NN^{\op}$.

On the source, for each positive integer $N$, one has the compact generator $\bigvee_{m|N}S^{\langle m\rangle}$, and on the target, one has the compact generator $\bigvee_{m|N}\Sigma^{\infty}_+(T/C_m)$. These compact generators are compatible with the diagrams. Hence, in light of the Schwede--Shipley theorem \cite[Th. 7.1.2.1]{HA}, it suffices to construct a natural equivalence of spectra
\[\equivto{S^{\langle m\rangle}(\langle n\rangle)}{F(\Sigma^{\infty}_+(T/C_{m}),\Sigma^{\infty}_+(T/C_{n}))\simeq\Sigma^{\infty}_+(T/C_{m})\wedge_{\Sigma^{\infty}_+T}\Sigma^{\infty}_+(T/C_{n})}\]
for any positive integers $m$ and $n$. This now follows from Ex. \ref{exm:representableMack} and the Segal--tom Dieck splitting theorem.
\end{proof}
\end{thm}

\begin{nul} Write $\categ{SymmSp}$ for the stable model category of symmetric spectra in simplicial sets, and let $N\in\widehat{\NN}$. We regard $A^{\eff}(\FF_{\cop_N})$ as a simplicial category. Now we may endow the category $\Fun(A^{\eff}(\FF_{\cop_N}),\categ{SymmSp})$ of simplicial functors
\[\fromto{A^{\eff}(\FF_{\cop_N})}{\categ{SymmSp}}\]
with its projective model structure, which is a left proper, combinatorial model category. It follows immediately from Th. \ref{thm:cyclonesareTspectra} that there is a zigzag of Quillen equivalences between the model category of orthogonal $T$-spectra relative to the family $\mathcal{F}_N$ and the left Bousfield localization of $\Fun(A^{\eff}(\FF_{\cop_N}),\categ{SymmSp})$ at the set of maps
\[H\coloneq\{\fromto{\Sigma_+^{\infty}\circ h_{X}\vee\Sigma_+^{\infty}\circ h_{X}}{\Sigma_+^{\infty}\circ h_{X\oplus Y}}\ |\ X,Y\in A^{\eff}(\FF_{\cop_N})\}.\]
\end{nul}

\begin{ntn} Let us write $DA(\FF_{\cop})$ for the nonabelian derived $\infty$-category of $A^{\eff}(\FF_{\cop})$. This is the $\infty$-category of functors $\fromto{A^{\eff}(\FF_{\cop})^{\op}}{\Top}$ that preserve products. It has the formal property that for any $\infty$-category $D$ that admits all colimits, the Yoneda embedding induces an equivalence
\[\equivto{\Fun^L(DA(\FF_{\cop}),D)}{\Fun^{\sqcup}(A^{\eff}(\FF_{\cop}),D)},\]
where $\Fun^{\sqcup}$ is the $\infty$-category of functors that preserve finite coproducts.
\end{ntn}

\begin{prp}\label{prp:cyclonicistensor} Suppose $N\in\NNhat$. For any presentable additive $\infty$-category $\AA$, one has an equivalence
\[\AA_{\cop_N}\simeq DA(\FF_{\cop_N})\otimes\AA.\]
Consequently, if $\AA$ is stable, then one has
\[\AA_{\cop_N}\simeq\Sp_{\cop_N}\otimes\AA.\]
\begin{proof} If $\AA$ is presentable and additive, then in light of the duality equivalence
\[D\colon A^{\eff}(\FF_{\cop_N})^{\op}\simeq A^{\eff}(\FF_{\cop_N}),\]
we have
\begin{eqnarray}
DA(\FF_{\cop_N})\otimes\AA&\simeq&\Fun^R(DA(\FF_{\cop})^{\op},\AA)\nonumber\\
&\simeq&\Fun^L(DA(\FF_{\cop_N}),\AA^{\op})^{\op}\nonumber\\
&\simeq&\Fun^{\oplus}(A^{\eff}(\FF_{\cop_N}),\AA^{\op})^{\op}\nonumber\\
&\simeq&\Fun^{\oplus}(A^{\eff}(\FF_{\cop_N})^{\op},\AA)\nonumber\\
&\simeq&\Fun^{\oplus}(A^{\eff}(\FF_{\cop_N}),\AA)\simeq\AA_{\cop_N}.\nonumber
\end{eqnarray}

If $\AA$ is moreover stable, then we can stabilize $DA(\FF_{\cop_N})$: the natural functor
\[\fromto{DA(\FF_{\cop})}{\Sp(DA(\FF_{\cop}))\simeq\Sp\otimes DA(\FF_{\cop})}\]
induces an equivalence
\[\equivto{\Fun^L(\Sp\otimes DA(\FF_{\cop}),\AA)}{\AA_{\cop}}.\]
Now
\[DA(\FF_{\cop_N})\otimes\Sp\simeq\Sp_{\cop_N},\]
and consequently, if $\AA$ is stable and admits all colimits, then
\[\Fun^L(\Sp_{\cop},\AA)\simeq\AA_{\cop},\]
whence (using duality again)
\[\AA_{\cop}\simeq((\AA^{\op})_{\cop})^{\op}\simeq\Fun^L(\Sp_{\cop},\AA^{\op})^{\op}\simeq\Fun^R(\Sp_{\cop}^{\op},\AA)\simeq\Sp_{\cop}\otimes\AA,\]
as desired.
\end{proof}
\end{prp}

\begin{nul} Fix, for the remainder of this section, $N\in\NNhat$. We can define a symmetric monoidal structure on $N$-cyclonic spectra via the second author's Day convolution, but in order to do so, we need to take a little care. The $2$-category $\FF_{\cop_N}$ does not admit products, so the effective Burnside $2$-category $A^{\eff}(\FF_{\cop_N})$ only admits the structure of a symmetric promonoidal $\infty$-category, denoted $A^{\eff}(\FF_{\cop_N})^{\circledast}$ in \cite{mack2}. Nevertheless, as we explain in \cite{mack2}, we are still entitled to employ the second author's Day convolution to get a symmetric monoidal structure on $\Sp_{\cop_N}$.
\end{nul}

\begin{dfn} Suppose $\AA^{\otimes}$ a presentable symmetric monoidal $\infty$-category (in the sense of \cite[Df. 3.4.4.1]{HA}, so that the symmetric monoidal structure preserves all colimits separately in each variable). Then the localized Day convolution yields the symmetric monoidal $\infty$-category
\[\AA_{\cop_N}^{\otimes}\coloneq\Mack(\FF_{\cop_N},\AA)^{\otimes}.\]
When $\AA^{\otimes}=\Sp^{\wedge}$, we call this the \emph{smash product} symmetric monoidal structure on cyclonic spectra.
\end{dfn}

\begin{nul} In the situation above, the tensor product functor on $\AA_{\cop_N}$ preserves colimits separately in each variable. So $\AA_{\cop_N}^{\otimes}$ is a presentable symmetric monoidal $\infty$-category. When $\AA=\Sp$, one finds that the homotopy category $h\Sp_{\cop_N}$ is naturally a triangulated tensor category, as conjectured by Kaledin. 
\end{nul}

\begin{exm} Thanks to our Barratt--Priddy--Quillen theorem \cite{mack2}, we also have an expression for the unit of the smash product symmetric monoidal structure. It is the Mackey functor
\[\SS\colon\fromto{A^{\eff}(\FF_{\cop_N})}{\Sp}\]
that carries a $C_N$-orbit $\angs{m}_N$ to the direct sum $K$-theory of
\[(\FF_{\cop_N})_{/\angs{m}_N}\simeq(\FF_{C_N})_{/\angs{m}_N};\]
this $K$-theory of course is nothing more than the wedge
\[\SS\angs{m}\simeq\bigvee_{n\in\NN_m}\Sigma^{\infty}_+B\Aut_{(\OO_{C})_{/\angs{m}}}(\angs{n})\simeq\bigvee_{n\in\NN_m}\Sigma^{\infty}_+B\left(\frac{1}{m}\ZZ\Big/\frac{1}{n}\ZZ\right).\]

On $\pi_0$, we obtain the unit in $\Ab_{\cop_N}$: it is the functor $\Omega\colon\fromto{A^{\eff}(\FF_{\cop_N})}{\Ab}$ that carries $\angs{m}_N$ to the Burnside ring $\Omega\angs{m}_N$ of finite $\frac{1}{m}\ZZ/\ZZ$-sets. Thus $\Omega\angs{m}\cong\ZZ\{\NN_m\}$, and in fact, one obtains a unique isomorphism
\[\Omega\cong\WW(\ZZ)\]
in $\categ{CAlg}(\Ab_{\cop_N}^{\otimes})$. This fact was first observed (and generalized) by Dress and Siebeneicher \cite{MR947758}.
\end{exm}


\section{Cyclotomic spectra} In this section, let us fix an supernatural number $N\in\NNhat$. We define $N_d$ to be the divisor of $N$ with the property that
\[v_p(N_d)\coloneq\begin{cases}
\infty&\text{if }v_p(N)=\infty;\\
0&\text{otherwise.}
\end{cases}\]
This corresponds to the divisible part of the abelian group $\frac{1}{N}\ZZ$. Accordingly, if $N$ is finite, then $N_d=1$, and the story of this section becomes uninteresting.

\begin{dfn} \label{dfn:mu_n} Now suppose $n$ a natural number that divides $N_d$. We define an endofunctor $\iota_n$ of $\OO_{\cop_N}$ as follows. On objects and $1$-morphisms, $\iota_n$ is given by pullback of the $C_N$-action along the multiplication-by-$n$ map $C_N\to C_N$, so that $\iota_n (\angs{m}_N)=\angs{mn}_N$. On $2$-isomorphisms, $\iota_n$ is the assignment $\goesto{r}{r/n}$. We note that $\iota_n$ is (strictly) fully faithful, and it identifies $\OO_{\cop_N}$ with the full subcategory of itself spanned by those objects $\angs{s}_N$ such that $n$ divides $s$.
\end{dfn}

\begin{lem} The functor $\iota_n$ extends to an essentially unique fully faithful functor 
\[\iota_n\colon\fromto{\FF_{\cop_N}}{\FF_{\cop_N}}\]
that preserves coproducts. This functor preserves pullbacks. Additionally, this functor admits a right adjoint $p_n\colon\fromto{\FF_{\cop_N}}{\FF_{\cop_N}}$ given by
\[p_n(\angs{m}_N) = \begin{cases} 
\angs{m/n}_N & \text{if }n\text{ divides }m \\
\emptyset & \text{otherwise.} \end{cases}
\]
In particular, $\iota_n$ and $p_n$ together exhibit $\OO_{\cop_N}$ as a localization of itself. In the language of \cite{BDGNS}, the functor $\iota_n$ is both \emph{left orbital} and \emph{right orbital}.
\begin{proof} If $n$ divides $s$ and $t$, then of course it divides their greatest common divisor. Consequently, the essential image of the fully faithful functor $\iota_n\colon \FF_{\cop_N} \to \FF_{\cop_N}$ is closed under pullbacks.

To see that the right adjoint $p_n$ exists, we observe that if it so happens that the pullback functor on presheaf categories
\[\iota_n^{\star} \colon \Fun(\OO_{\cop_N}^{\op},\Top) \to \Fun(\OO_{\cop_N}^{\op},\Top)\]
preserves the full subcategory $\FF_{\cop_N}$, then the desired right adjoint is $\iota_n^{\star}|_{\FF_{\cop_N}}$. To verify that this occurs, we must check that $\iota_n^{\star}$ carries representables to coproducts of representables. But 
\[\iota_n^{\star} \Map(-, \angs{m}_N) = \begin{cases} 
\Map(-, \angs{m/n}_N) & \text{if }n\text{ divides }m; \\
\emptyset & \text{otherwise.} \end{cases}\]
Hence we have our desired right adjoint.
\end{proof}
\end{lem}

\begin{ntn} In light of the previous result, if $n\in\NN_{N_d}$, then we find that both $\iota_n$ and $p_n$ each induce direct-sum-preserving endofunctors -- which we denote $A^{\eff}(\iota_n)$ and $A^{\eff}(p_n)$ -- on the effective Burnside $\infty$-category $A^{\eff}(\FF_{\cop_N})$. It is not true that these are adjoint to each other, but it is the case that one has a natural equivalence $A^{\eff}(p_n)A^{\eff}(\iota_n)\simeq\id$.
\end{ntn}

\begin{ntn} Suppose $N\in\NNhat$ and $n\in\NN_{N_d}$. For any presentable $\infty$-category $\AA$, the functor $A^{\eff}(p_n)$ induces a restriction endofunctor
\[i_{n,\star}\coloneq A^{\eff}(p_n)^{\star}\colon\fromto{\AA_{\cop_N}}{\AA_{\cop_N}}\]
that commutes with both limits and colimits. Consequently, it admits a left adjoint $i_n^{\star}$, which is given by left Kan extension along $A^{\eff}(p_n)$. (Recall \cite[Lemma 2.20]{aroneching} that the Kan extension of an additive functor between semiadditive categories along an additive functor between semiadditive categories is always additive.) It also admits a right adjoint $i_n^!$, given by right Kan extension along $A^{\eff}(p_n)$, for which we will have less need. (All of this notation is chosen to suggest an analogy with the pushforward of sheaves along a closed immersion.)

When $\AA=\Sp$, the functors $i_n^{\star}$ correspond under the equivalence of Th. \ref{thm:cyclonesareTspectra} to the geometric fixed point functors for $C_n\subset T$; indeed, they preserve colimits, so it suffices to check this on corepresentables (i.e., on sphere spectra), where it is obvious.
\end{ntn}

\begin{cnstr} Suppose now that $p\in\NN_{N_d}$ is prime. For any supernatural number $M$, let us write $M(p')$ for the prime-to-$p$ part of $M$; i.e., $M(p')$ is the supernatural number with
\[v_q(M(p'))=\begin{cases}
v_q(M)&\text{if }q\neq p;\\
0&\text{if }q=p.
\end{cases}
\]
Now let $j_p\colon\fromto{\OO_{\cop_{N(p')}}}{\OO_{\cop_N}}$ be the complementary sieve to the cosieve $i_p$; that is, $j_p$ is the inclusion of the full subcategory spanned by those $\angs{m}_N$ such that $p$ does not divide $m$. As a sieve, this functor extends to a pullback-preserving functor $\fromto{\FF_{\cop_{N(p')}}}{\FF_{\cop_N}}$, whence we obtain a fully faithful functor
\[A^{\eff}(j_p)\colon A^{\eff}(\FF_{\cop_{N(p')}})\to A^{\eff}(\FF_{\cop_N}).\]

Now if $\AA$ is an additive presentable $\infty$-category, then we may thus contemplate the restriction functor
\[j_p^{\star}\colon\fromto{\AA_{\cop_N}}{\AA_{\cop_{N(p')}}}\]
and its fully faithful left and light adjoints
\[j_{p,!}\text{\quad and\quad}j_{p,\star}\colon\fromto{\AA_{\cop_{N(p')}}}{\AA_{\cop_N}}\]
along $A^{\eff}(j_p)$. The left adjoint may be viewed as the ``$p$-constant cyclonic object'' functor; that is, 
\begin{itemize}
\item For any $n\in\NN_N$ and any $X \in \AA_{\cop_{N(p')}}$,
\[j_{p,!}X\angs{n}_N\coloneq X\angs{n(p')}_{N(p')};\]
\item For any $m,n \in\NN_N$ such that $m$ divides $n$ and for any $X \in \AA_{\cop_{N(p')}}$, the morphism
\[\phi_{m | n, \star}\colon j_{p,!}X\angs{m}_N \to j_{p,!}X \angs{n}_N \]
is the identity if $n / m$ is a power of $p$, and is $\phi_{m | n, \star}\colon\fromto{X\angs{m}_N}{X\angs{n}_X}$ if $n/m$ is coprime to $p$;
\item For any $m, n \in\NN_N$ such that $m$ divides $n$ and for any $X \in \AA_{\cop_{N(p')}}$,
\[\phi_{m |n }^\star\colon j_{p,!}X \angs{n}_N \to j_{p,!}X \angs{m}_N \]
is multiplication by $n / m$ if $n/m$ is a power of $p$, and is $\phi_{m | n}^\star$ if $n/m$ is prime to $p$.
\end{itemize}
\end{cnstr}

\begin{prp} Suppose $\AA$ a presentable stable $\infty$-category, and suppose $p\in\NN_{N_d}$ a prime. Then the two fully faithful functors
\[i_{p,\star}\colon\into{\AA_{\cop_N}}{\AA_{\cop_N}}\text{\quad and\quad}j_{p,!}\colon\into{\AA_{\cop_{N(p')}}}{\AA_{\cop_N}}\]
together exhibit $\AA_{\cop_N}^{\op}$ as a recollement of $\AA_{\cop_N}^{\op}$ and $\AA_{\cop_{N(p')}}^{\op}$ \cite[Df. A.8.1]{HA}; equivalently, they give a stratification of $\AA_{\cop_N}^\op$ along $\Delta^1$ \cite[Definition 3.4]{aroneching}.
\begin{proof} This follows immediately from \cite[Proposition 3.13]{aroneching}.
\end{proof}
\end{prp}

\begin{nul}\label{nul:normcofseq} In particular \cite[Theorem 2.32]{aroneching}, we have a cofiber sequence
\[j_{p,!}j_p^*X\to X\to i_{p,\star}i_p^{\star}X.\]
\end{nul}

\begin{nul} Now the product of any two elements of $\NN_{N_d}$ again lies in $\NN_{N_d}$. Hence we may consider the multiplicative monoid $\NN_{N_d}$, and the endofunctors $\iota_n$ fit together into a functor
\[I\colon B\NN_{N_d} \to \Cat_\iy.\]
To see this, it's best to construct the cartesian fibration $p \colon \OO_{\cop_N}^{\para} \to B\NN_{N_d}$ classified by $I$.
\end{nul}

\begin{dfn} Let $\OO_{\cop_N}^{\para}$ be the $2$-category in which
\begin{itemize}
\item the objects are those of $\OO_{\cop_N}$;
\item a $1$-morphism from $S$ to $T$ is a pair $(m,f)$ consisting of an element $m \in \NN_{N_d}$ and a morphism $f \colon \iota_m(S) \to T$ in $\OO_{\cop_N}$, with composition law
\[(n, g) \circ (m, f) = (nm,  g \iota_n(f))\]
using the natural isomorphism $ \iota_n \iota_m \cong  \iota_{nm}$;
\item between two $1$-morphisms $(m,f\colon \iota_m(S) \to T)$ and $(n,g\colon \iota_n(S) \to T)$ the morphism is given by
\[\Mor_{\Mor_{\OO_{\cop_N}^{\para}}(S,T)}((m, f), (n, g))\coloneq\begin{cases}
\Mor_{\Mor_{\OO_{\cop_N}}( \iota_m(S),T)}(f, g)&\text{if }m=n\\
\varnothing&\text{if }m\neq n;
\end{cases}\]
\item the composition functor
\[\Mor((m , f_1), (m, f_2)) \X \Mor((n, g_1), (n, g_2)) \to \Mor((nm,  g_1 \iota_n(f_1)), (nm,  g_2 \iota_n(f_2)))\]
is given by
\[\goesto{(r, s)}{\frac{r}{m}  + s}.\]
\end{itemize}
\end{dfn}

We would like to demonstrate that $p\colon \OO_{\cop_N}^{\para}\to B\NN_{N_d}$ is a cocartesian fibration. To this end, we must digress in order to make some technical observations about the lifting problems in $2$-categories.

By construction, any $2$-category, when regarded as an $\infty$-category using Cnstr. \ref{cnstr:NCDelta} is a $3$-coskeletal simplicial set as well as a $2$-category in the sense of \cite[\S 2.3.4]{HTT}. That implies that if $f\colon\fromto{\CC}{\DD}$ is a functor between two $2$-categories, then for any solid arrow diagram
\begin{equation*}
\begin{tikzpicture}
\matrix(m)[matrix of math nodes,
row sep=4ex, column sep=4ex,
text height=1.5ex, text depth=0.25ex]
{\Lambda^n_k & \CC \\
\Delta^n & \DD, \\ };
\path[>=stealth,->,font=\scriptsize]
(m-1-1) edge node[above]{$\eta$} (m-1-2)
edge[right hook->] (m-2-1)
(m-1-2) edge (m-2-2)
(m-2-1) edge node[below]{$\theta$} (m-2-2)
edge[dotted] (m-1-2);
\end{tikzpicture}
\end{equation*}
a dotted lift exists automatically either if $n\geq 5$ or if $n\geq 3$ and $1\leq k\leq n-1$ \cite[Pr. 2.3.4.7]{HTT}. We therefore undwind this finite check to give the following criterion.

\begin{lem}\label{lem:21functorisinnerfib} A functor $f\colon\fromto{\CC}{\DD}$ between two $2$-categories is an inner fibration if and only if, for any $2$-isomorphism $\lambda\colon\equivto{\psi}{\chi}$ of $\DD$ and for any lift $\upsilon$ of $\psi$ to $\CC$, there exists a $2$-isomorphism $\kappa\colon\equivto{\upsilon}{\phi}$ of $\CC$ lifting $\lambda$.
\end{lem}

\noindent The condition to be a cocartesian fibration is a bit more involved, but it's no less explicit.

\begin{lem}\label{lem:21functoriscocartesian} Suppose $f\colon\fromto{\CC}{\DD}$ an inner fibration. Then a $1$-morphism $\phi$ of $\CC$ is $f$-cocartesian if and only if, for any $2$-isomorphism $\lambda\colon\equivto{\alpha\circ f(\phi)}{\beta}$ and any $1$-morphism $\chi$ covering $\beta$, there exist a $1$-morphism $\psi$ covering $\alpha$ and a $2$-isomorphism $\kappa\colon\equivto{\psi\circ\phi}{\chi}$ covering $\lambda$ such that for any $1$-morphism $\psi'$ covering $\alpha$ and any $2$-isomorphism $\kappa'\colon\equivto{\psi'\circ\phi}{\chi}$ covering $\lambda$, there exists a unique $2$-isomorphism $\zeta\colon\equivto{\psi}{\psi'}$ covering $\id_{\alpha}$ such that $\kappa=\kappa'\circ(\zeta\ast\phi)$.
\end{lem}

It now follows easily from Lms. \ref{lem:21functorisinnerfib} and \ref{lem:21functoriscocartesian} that the projection
\[p\colon \OO_{\cop_N}^{\para}\to B\NN_{N_d}\]
is a cocartesian fibration, and it's immediate from the construction that it's the one that is classified by $I$.

\begin{dfn} We define the \emph{relative presheaf category of $\OO_{\cop_N}^{\para}$ over $B\NN_{N_d}$} as the simplicial set $\PP_{\cop_N}^{\para}$ over $(B\NN_{N_d})^\op$ determined by the existence of bijections
\[\Mor_{(B\NN_{N_d})^\op}(K, \PP_{\cop_N}^{\para}) \cong \Mor(K \X_{(B\NN_{N_d})^\op} (\OO_{\cop_N}^{\para})^\op, \Top)\]
natural in $K \in s\Set_{(/B\NN_{N_d})^\op}$. We regard $\PP_{\cop_N}^{\para}$ as a simplicial set over $B \NN$ by applying the canonical isomorphism
\[(B\NN_{N_d})^\op \cong B\NN_{N_d}.\]
\end{dfn}

Now by \cite[Pr. 3.2.2.13]{HTT}, we see that the structure map
\[\PP_{\cop_N}^{\para} \to B\NN_{N_d}\]
is a cocartesian fibration whose fiber over the unique object of $B\NN_{N_d}^{\X}$ is the presheaf $\iy$-category $\Fun(\OO_{\cop_N}^{\op},\Top)$. For any edge $n$ of $B\NN_{N_d}$, the pushforward functor is equivalent to the functor $\mu_n^{\star}$.

\begin{dfn} Let $\FF_{\cop_N}^{\para}$ be the full subcategory of $\PP_{\cop_N}^{\para}$ spanned by the objects of $\FF_{\cop_N}$. Since each $\mu_n^{\star}$ carries respresentables to coproducts of representables, the projection
\[\pi\colon \FF_{\cop_N}^{\para}\to B\NN_{N_d}\]
is again a cocartesian fibration, and the cocartesian pushforward over $n \in B\NN_{N_d}$ can be identified with $i_n$.
\end{dfn}

Now we need to use the functoriality of the assignment $\goesto{C}{A^{\eff}(C)}$; since we aim to construct our $\infty$-categories vertically (i.e., using fibration of various kinds), we shall describe this functor vertically as well.

\begin{dfn} Let $\twa(\D^n)$ be the twisted arrow category of the $n$-simplex \cite[\S 2]{mack1}; its vertices are labeled by pairs of integers $(i, j)$ with $0 \leq i \leq j \leq n$. Let
\[D^n\subseteq\twa(\D^n)^{\op} \X \D^n\]
be the full subcategory spanned by those triples $((i, j), h)$ for which $0 \leq i \leq j \leq h \leq n$, and let $\rho\colon D^n \to \D^n$ be the projection onto the last factor. Suppose $q\colon E \to B$ is a cocartesian fibration of $\iy$-categories. We define the \emph{effective Burnside category of $E$ relative to $B$}, $A^{\eff}_B(E)$, as the simplicial set whose $n$-simplices over a fixed $n$-simplex $\sigma\colon\D^n \to B$ are maps $z\colon D^n \to E$ such that
\begin{itemize}
\item the diagram
\[\begin{tikzcd}
D^n \ar{r}{z} \ar{d}{\rho} & E \ar{d}{q} \\
\D^n \ar{r}{\sigma} & B \end{tikzcd}\]
commutes;
\item for any $((i, j), h)$ with $h \leq n - 1$, the image under $z$ of the edge 
\[((i, j), h) \to ((i, j), h + 1)\]
 is $q$-cocartesian;
\item for any $i_0, i_1, j_0, j_1, h$ with
\[0 \leq i_0 \leq i_1 \leq j_1 \leq j_0 \leq h,\]
the image under $z$ of the square spanned by the four vertices
\[((i_0, j_0), h), ((i_0, j_1), h), ((i_1, j_0), h), ((i_1, j_1), h),\]
which necessarily lies within a fiber of $q$, is a pullback square.
\end{itemize}
The following is proved in \cite[Appendix C]{BDGNS}:
\end{dfn}

\begin{prp} If the fibers of $E$ admit pullbacks, and those pullbacks are preserved by the cocartesian pushforwards, then $A^{\eff}_B(E)$ is a cocartesian fibration whose cocartesian edges are given by those maps $z : D^1 \to E$ which take the edges
\[((0, 1), 1) \to ((0, 0), 1), ((0, 1), 1) \to ((1, 1), 1)\]
to equivalences.
\end{prp}

\begin{nul}\label{nul:unpackAeffrel} The relative Burnside category of interest to us will be $A^{\eff}_{B\NN_{N_d}}(\FF_{\cop_N}^{\para})$. It is a cocartesian fibration whose fiber over the unique vertex of $B\NN_{N_d}$ is $A^{\eff}(\FF_{\cop_N})$ and whose cocartesian pushforward over $n$ is equivalent to $A^{\eff}(j_n)$.

More explicitly, $A^{\eff}_{B\NN_{N_d}}(\FF_{\cop_N}^{\para})$ can be identified with the following $2$-category:
\begin{itemize}
\item An object is a finite cyclonic set of degree $N$, hence a disjoint union of cyclonic orbits.
\item For any finite cyclonic sets $X$ and $Y$ of degree $N$, a $1$-morphism $\fromto{X}{Y}$ is a pair $(n,f)$ consisting of an element $n\in\NN_{N_d}$ and a span $f$ from $X$ to $\mu_nY$ of the form
\begin{equation*}
\begin{tikzpicture}[baseline]
\matrix(m)[matrix of math nodes, 
row sep={7ex,between origins}, column sep={7ex,between origins}, 
text height=1.5ex, text depth=0.25ex] 
{&\mu_nU&\\ 
X&&\mu_nY;\\}; 
\path[>=stealth,->,font=\scriptsize,inner sep=0.5pt] 
(m-1-2) edge node[above left]{$f_0$} (m-2-1) 
edge node[above right]{$\mu_n(f_1)$} (m-2-3); 
\end{tikzpicture}
\end{equation*}
\item A $2$-isomorphism $\equivto{(n,f)}{(n',f')}$ between two $1$-morphisms
\[\fromto{X}{Y}\]
can only exist if $k=k'$, in which case it is a diagram
\begin{equation*}
\begin{tikzpicture}[baseline]
\matrix(m)[matrix of math nodes, 
row sep={7ex,between origins}, column sep={7ex,between origins}, 
text height=1.5ex, text depth=0.25ex] 
{&&[-3ex]\mu_nU&[-3ex]&\\ 
X&s&&t&\mu_nY\\
&&\mu_nU'&&\\}; 
\path[>=stealth,->,font=\scriptsize,inner sep=0.75pt] 
(m-1-3) edge node[above left]{$f_0$} (m-2-1)
edge node[right]{$\phi$} (m-3-3)
edge node[above right]{$\mu_n(f_1)$} (m-2-5) 
(m-3-3) edge node[below left]{$f'_0$} (m-2-1) 
edge node[below right]{$\mu_n(f'_1)$} (m-2-5); 
\end{tikzpicture}
\end{equation*}
in which $\phi$ is an isomorphism, and $s$ and $t$ are intertwiners.
\end{itemize}
\end{nul}

\begin{dfn} Suppose $\AA$ a presentable, additive $\infty$-category. Let $\AA^{\para\prime}_{\cop_N}$ be the unique simplicial set over $B\NN_{N_d}$ giving natural bijections
\[\Mor_{B\NN_{N_d}}(K, \AA^{\para\prime}_{\cop_N}) \cong \Mor(K \X_{B\NN_{N_d}} A^{\eff}_{B\NN_{N_d}}(\FF_{\cop_N}^{\para}), \AA)\]
for $K \in s\Set_{/B\NN_{N_d}}$. By the opposite of \cite[Pr. 3.2.2.13]{HTT}, the projection 
\[\rho \colon \AA^{\para\prime}_{\cop_N} \to B\NN_{N_d}\]
is a cartesian fibration. A vertex of $\AA^{\para\prime}_{\cop_N}$ is a functor $A^{\eff}(\FF_{\cop_N})\to\AA$, and the cartesian pullback over the edge $n$ of $B\NN_{N_d}$ may be identified with the precomposition with
\[A^{\eff}(p_n)\colon A^{\eff}(\FF_{\cop_N}) \to A^{\eff}(\FF_{\cop_N}).\]
Let $\AA^{\para}_{\cop_N}$ be the full subcategory of $\AA^{\para\prime}_{\cop_N}$ spanned by the cyclonic spectra.
\end{dfn}

\begin{nul} Since all of the $A^{\eff}(p_n)$ are additive functors, the cartesian pullbacks preserve this full subcategory, and so the projection
\[ \pi\colon\AA^{\para}_{\cop_N} \to B\NN_{N_d}\]
is a cartesian fibration. Moreover, each of the cartesian pullbacks admits a left adjoint given by the left Kan extension $i_n^{\star}$ along $A^{\eff}(p_n)$, so $\pi$ is also cocartesian.
\end{nul}

\begin{nul} Consequently, for an $N$-cyclonic object $X\colon\fromto{A^{\eff}(\FF_{\cop_N})}{\AA}$, we can express the values of $i_n^{\star}X$ explicitly as a certain (admittedly rather complicated) colimit:
\[i_n^{\star}X\angs{m}_N\simeq\underset{W\in A(n,\angs{m}_N)}{\colim}\ X(W)\]
is an equivalence, where
\[A(n,\angs{m}_N)\subset A^{\eff}_{B\NN_{N_d}}(\FF_{\cop_N}^{\para})_{/\angs{m}_N}\]
is the subcategory whose objects are morphisms $\fromto{W}{\angs{m}_N}$ that cover the edge $n\in(B\NN_{N_d})_1$ and whose morphisms are morphisms $\fromto{W'}{W}$ over $\angs{m}_N$ whose image in $(B\NN_{N_d})_1$ is degenerate.
\end{nul}

\begin{dfn} Suppose $\AA$ a presentable, additive $\infty$-category, and suppose $N\in\NNhat$. Then the $\infty$-category of \emph{$N$-cyclotomic objects of $\AA$} is the $\infty$-category
\[\AA_{\Phi_N}\coloneq\Map^{\flat}_{s\Set^+_{/B\NN_{N_d}}}((B\NN_{N_d})^{\sharp},(\AA^{\para}_{\cop_N})^{\natural})\]
of cocartesian sections of the cocartesian fibration $\pi$. Similarly, the $\infty$-category of \emph{$N$-pre\-cyc\-lo\-tomic objects of $\AA$} is the $\infty$-category
\[\AA_{\Phi'_N}\coloneq\Map_{s\Set_{/B\NN_{N_d}}}(B\NN_{N_d},\AA^{\para}_{\cop_N})\]
of \emph{all} sections of $\pi$.

More conceptually, $\AA_{\Phi_N}$ is the limit of the diagram 
\[B\NN_{N_d} \to \Cat_\iy\]
classifying $\pi$, and $\AA_{\Phi'_N}$ is the lax limit of this diagram.

When $N=\infty$, we shall drop the $N$, write $\AA_{\Phi}$, and refer just to \emph{cyclotomic} and \emph{precyclotomic objects}.
\end{dfn}

\begin{nul} We unwind the definitions to find that $\AA_{\Phi_N}$ is the full subcategory of the $\infty$-category $\Fun(A^{\eff}_{B\NN_{N_d}}(\FF_{\cop_N}^{\para}),\AA)$ spanned by those functors
\[X\colon\fromto{A^{\eff}_{B\NN_{N_d}}(\FF_{\cop_N}^{\para})}{\AA}\]
such that:
\begin{itemize}
\item the restriction of $X$ to $A^{\eff}(\FF_{\cop_N})\subset A^{\eff}_{B\NN_{N_d}}(\FF_{\cop_N}^{\para})$ is a Mackey functor, and
\item for any orbit $\angs{m}_N$ and any $n\in\NN_{N_d}$, the natural map
\[\fromto{i_n^{\star}X\angs{m}_N\simeq\underset{W\in A(n,\angs{m}_N)}{\colim}\ X(W)}{X\angs{m}_N}\]
is an equivalence, where $A(n,\angs{m}_N)$ is as above.
\end{itemize}
\end{nul}

\begin{nul} There are two ways to think of the universal property of $\AA_{\Phi_N}$, because it is the limit of presentable $\infty$-categories and left adjoints, but at the same time it is the colimit of presentable $\infty$-categories and right adjoints. So on one hand, the assertion is that for any presentable $\infty$-category $C$, any colimit-preserving functor $F\colon\fromto{C}{\AA_{\cop_N}}$, and any suitably compatible collection of natural equivalences $\equivto{i_n^{\star}\circ F}{F}$, one has an essentially unique factorization of $F$ through $\AA_{\Phi_N}$. On the other hand, for any presentable $\infty$-category $D$, any limit-preserving, accessible functor $G\colon\fromto{\AA_{\cop_N}}{D}$, and any suitably compatible collection of natural equivalences $\equivto{G}{G\circ A^{\eff}(i_n)^{\star}}$, one has an essentially unique factorization of $G$ through $\AA_{\Phi_N}$.
\end{nul}

\begin{thm} The $\infty$-category $\Sp_{\Phi}$ is equivalent to the underlying homotopy theory of the Blumberg--Mandell relative category of cyclotomic spectra, and the $\infty$-category $\Sp_{\Phi_{p^{\infty}}}$ is equivalent to the underlying homotopy theory of the Blumberg--Mandell relative category of $p$-cyclotomic spectra
\begin{proof} Suppose $N\in\NNhat$. Under the equivalence between $\Sp_{\cop_N}$ and $\categ{OrthoSp}_{T,N}$ of Th. \ref{thm:cyclonesareTspectra}, the functors $i_n^{\star}$ for $n\in\NN_{N_d}$ agree with the functors $\rho_n^{\star}\Phi^{C_n}$ of \cite{BM}. Therefore, when $N=\infty$ or $N=p^{\infty}$, it suffices to show that the ``model${}^{\ast}$ categories'' of \cite{BM} do indeed model the homotopy limit of this diagram of homotopy theories. This follows from the following lemma. 
\end{proof}
\end{thm}

\begin{lem} Suppose $D$ a small category, and suppose $\FF\colon\fromto{D}{\categ{RelCat}}$ a diagram of a relative categories and relative functors (in the sense of \cite{relcat}). Suppose, additionally, that $\FF_d$ is a partial model category \cite[1.1]{partmodcat} for any object $d\in D$. Then the homotopy limit of the diagram $N\FF$ of underlying $\infty$-categories is equivalent to the relative nerve of the category of cocartesian sections of $\FF$,-- i.e., the category of those tuples $(X,\phi)=((X_d)_{d\in D},(\phi_f)_{f\in\mathrm{Arr}(D)})$ comprised of an object $X_d$ of $\FF_d$, one for each object $d\in D$, and a weak equivalence $\phi_f\colon\fromto{F(f)X_d}{X_e}$, one for each morphism $f\colon\fromto{d}{e}$ of $D$, such that for any composable pair of morphisms $f$ and $g$ of $D$, one has
\[\phi_{f\circ g}=\phi_g\circ F(g)\phi_f,\]
in which a weak equivalence is a morphism $\fromto{(X,\phi)}{(Y,\psi)}$ such that for any object $d\in D$, the morphism $\fromto{X_d}{Y_d}$ is a weak equivalence of $\FF_d$.
\begin{proof} If the functors $F(f)$ were all left Quillen functors between model categories, we would be done by Bergner's \cite[Th. 4.1]{bergnerholim}. But in fact, the proof there applies immediately in this context, in light of \cite[3.3]{partmodcat}.
\end{proof}
\end{lem}

\begin{prp} Suppose $\AA$ a presentable, additive $\infty$-category. Then $\AA_{\Phi_N}$ is presentable and additive. Moreover, if $\AA$ is stable, then so is $\AA_{\Phi_N}$.
\begin{proof} The first claim follows from \cite[Pr. 5.5.3.13]{HTT} and an easy argument. The second claim follows from \cite[Pr. 4.8.2.18]{HA} and \cite[Th. 3.4.3.1]{HA}.
\end{proof}
\end{prp}

\begin{nul} So the $\infty$-category $\Sp_{\Phi_N}$ of $N$-cyclotomic spectra is a presentable stable $\infty$-category. This may be surprising, in light of Blumberg and Mandell's \cite{BM}, where only a ``model${}^{\ast}$ structure'' of cyclotomic spectra is constructed; one might have therefore have feared that it would be impossible to form general limits and colimits in $\Sp_{\Phi_N}$. In fact, the presentability of $\Sp_{\Phi_N}$ implies that there is a combinatorial model category of cyclotomic spectra.
\end{nul}

\begin{exm} Suppose $\AA$ a Grothendieck abelian $1$-category. An $N$-cyclotomic object of $\AA$ factors through the homotopy category $hA^{\eff}_{B\NN_{N_d}}(\FF_{\cop_N}^{\para})$. Thus it is the data of an $N$-cyclonic object $X$ (as described in Ex. \ref{exm:cyclonic1cat}) along with some additional structure that we now describe. For simplicity, we'll suppose that $N = N_d$, so all divisors of $N$ occur with infinite multiplicity, and $\frac{1}{N}\ZZ$ is divisible.

Let $\mc{D}(\AA)$ be the derived category of $\AA$, which by \cite[Proposition 1.3.5.9]{HA} and \cite[Proposition 1.3.5.21]{HA} is a stable presentable $\iy$-category with $t$-structure. Denote by $\mc{X}$ the $N$-cyclonic object of $\mc{D}(\AA)$ obtained from $X$ by including $\AA$ into $\mc{D}(\AA)$ as the heart. By \ref{nul:normcofseq}, for any prime $p \in N_d$, we have a fiber sequence of $N$-cyclonic objects
\[j_{p, !} j_p^{\star} \mc{X}  \to \mc{X} \to i_{p,\star} i_p^{\star} \mc{X} .\]
Since all of the functors in this sequence are right exact on $\AA_{\cop_N}$, taking $\pi_0$ gives a right exact sequence
\[j_{p, !} j_p^{\star} X \os{\lambda_p} \to X \to i_{p, \star} i_p^{\star} X \to 0\]
in $\AA_{\cop_N}$.
Unwinding the definitions, we have
\[j_{p, !} j_p^{\star} X \angs{ap^v} \cong X \angs{a}\]
and 
\[\lambda_p\angs{ap^v} = \phi_{p^v | ap^v, \star} \colon X \angs{a} \to X \angs{ap^v},\]
where $a$ is coprime to $p$. Thus
\[i_{p, \star} i_p^{\star} X \angs{ap^v} \cong X \angs{a p ^v} / X \angs{a}.\]
Since $i_{p, \star}$ simply acts by reparametrization, the $N$-cyclonic object $\Phi^{C_p} X \coloneq i_p^{\star} X$ of $C_p$-geometric fixed points is given by 
\[\Phi^{C_p} X \angs{ap^v} \cong X \angs{ap^{v + 1}} / X \angs{a}.\]
For any pair of primes $p_1, p_2 \in N_d$, there's a natural isomorphism
\[\Phi^{C_{p_1}} \Phi^{C_{p_2}} X \cong \Phi^{C_{p_2}} \Phi^{C_{p_1}} X\]
which we'll make implicit by referring to each of these objects as $\Phi^{p_1 p_2} X$.
A cyclotomic structure on $X$ is, for each prime $p \in N$, a cyclonic isomorphism
\[r_p\colon\equivto{X}{\Phi^{C_p} X}\]
such that for each pair of primes $p_1, p_2 \in N_d$, the diagram
\[\begin{tikzcd}
X \ar{r}{r_{p_1}} \ar{d}[left]{r_{p_2}}& \Phi^{C_{p_1}} X \ar{d}{r_{p_2}} \\
\Phi^{C_{p_2}} X \ar{r}[below]{r_{p_1}} & \Phi^{p_1 p_2} X
\end{tikzcd}\]
commutes.

Unpacking slightly, the isomorphism $r_p$ amounts to giving, for each $n = ap^v \in N$, an isomorphism of abelian groups
\[r_{n | pn}\colon X \angs{n} \cong X\angs{pn} / X\angs{a}\]
compatibly with transfer and Frobenius maps.

A cyclotomic structure gives rise to a third type of relation between the objects of $\AA$ which constitute an object $X \in \AA_{\cop_N}$: \emph{restriction maps}
\[\rho_{m | n} \colon X \angs{n} \to X\angs{m}\]
whenever $u, v \in N$ with $m | n$. When $n = mp = a p^v$ for some prime $p$, $\rho_{m | n}$ is defined as the composition
\[X \angs{n} \to X \angs{n} / X\angs{a} \os{r_{ m | n}^{-1}} \to X \angs{m}.\]
By the compatibility square above, we may safely extend to all $m$ and $n$ using the prescription
\[\rho_{k | m} \rho_{m | n} = \rho_{k | n}.\]
Restriction maps commute with transfer and Frobenius maps.
\end{exm}

Since $\pi_0\colon \Sp \to \Ab$ is an additive functor, it extends to a functor
\[\Sp_{\cop_N} \to \Ab_{\cop_N},\]
which we'll also denote $\pi_0$. We now have:

\begin{lem} \label{lem:pi0cyclotomic} $\pi_0$ extends to a functor from the category $\Sp_{\Phi_N}^{\geq 0}$ of \emph{connective} $N$-cyclotomic spectra to the category $\Ab_{\Phi_N}$ of $N$-cyclotomic abelian groups.
\end{lem}

\begin{proof}
Let $E \in \Sp_{\cop_N}$ be an $N$-cyclonic spectrum. Then we have a cofiber sequence
\[j_{p, !} j_p^{\star} E \to E \to i_{p, \star} \Phi^{C_p} E.\]
By connectivity, taking $\pi_0$ gives rise to a short exact sequence
\[j_{p, !} j_p^{\star} (\pi_0 E) \os{\lambda_p} \to \pi_0 E \to i_{p,\star} (\pi_0 \Phi^{C_p} E) \to 0.\]
We deduce that
\[\Phi^{C_p} \pi_0 E \cong \pi_0 \Phi^{C_p} E\]
and since an $N$-cyclotomic structure on $E$ amounts to a coherently compatible system of equivalences $\Phi^{C_p} E \simeq E$ as $p$ ranges over the primes dividing $N$, the result follows.
\end{proof}

\begin{exm}
Let $A$ be a commutative ring. Since we have an isomorphism of cyclonic abelian groups
\[\WW_\bullet(A) \cong \pi_0 \THH(A).\]
[See \cite[Th. 3.3]{MR1410465} for the $p$-typical case.] Since $\THH(A)$ is always connective for an ordinary commutative ring $A$, Lemma \ref{lem:pi0cyclotomic} implies that $\WW_\bullet(A)$ is naturally a cyclotomic abelian group. Unwinding the definitions, this comes down to giving isomorphisms
\[\WW_{\angs{ap^v}}(A) \cong \WW_{\angs{ap^{v + 1}}}(A) / V_{a | ap^{v + 1}} \WW_{\angs{a}}(A).\]
for each $a, p$ and $v$. But this is immediate from inspection at the level of Witt components.
\end{exm}

\begin{nul} We can also contemplate the symmetric monoidal structures on $\Sp_{\Phi_N}$. For this, one may construct a symmetric monoidal extension of the cocartesian fibration $\pi$ above: if $\AA^{\otimes}$ is a presentable, symmetric monoidal, additive $\infty$-category, then the cocartesian fibration $\pi$ constructed above extends to a cocartesian fibration
\[\pi^{\otimes}\colon\fromto{(\AA^{\para}_{\cop})^{\otimes}}{B\NN_{N_d}\times N\Lambda(\FF)}\]
whose pullback along $\fromto{N\Lambda(\FF)}{B\NN_{N_d}\times N\Lambda(\FF)}$ is symmetric monoidal, which exhibits a factorization of the functor $\fromto{B\NN_{N_d}}{\Cat_{\infty}}$ that classifies $\pi$ through a functor
\[\fromto{B\NN_{N_d}}{\categ{CAlg}(\categ{Pr}_{\infty}^{L,\otimes})}.\]
We leave the details to the reader, for now; we will return to this at a later point.

If we form the limit of the functor $\fromto{B\NN_{N_d}}{\categ{CAlg}(\categ{Pr}_{\infty}^{L,\otimes})}$, then by \cite[Pr. 3.2.2.1]{HA}, we obtain a presentable symmetric monoidal $\infty$-category $\AA_{\Phi_N}^{\otimes}$.

In particular, when $\AA=\Sp$, we obtain a presentable symmetric monoidal stable $\infty$-category $\Sp_{\Phi_N}^{\otimes}$. This lifts the triangulated tensor category structure Cary Malkiewich found \cite[Cor. 1.5]{Malk} on the homotopy category $h\Sp_{\Phi_N}^{\otimes}$.
\end{nul}


\section{Kaledin's conjecture} Kaledin seeks \cite[(0.1)]{MR3137194} a pullback square of ``noncommutative brave new schemes''
\begin{equation*}
\begin{tikzpicture}[baseline]
\matrix(m)[matrix of math nodes,
row sep=4ex, column sep=4ex,
text height=1.5ex, text depth=0.25ex]
{\Spec\DD(R)_{\Psi} & \Spec\Sp_{\Psi} \\
\Spec\DD(R) & \Spec\Sp \\ };
\path[>=stealth,->,font=\scriptsize]
(m-1-1) edge node[above]{} (m-1-2)
edge node[left]{} (m-2-1)
(m-1-2) edge node[right]{} (m-2-2)
(m-2-1) edge node[below]{} (m-2-2);
\end{tikzpicture}
\end{equation*}
for any commutative ring $R$ that induces the square \eqref{eqn:Kaledinsquare} under formation of quasicoherent modules. Let us make this precise.

\begin{dfn} The $\infty$-category of \emph{noncommutative affine schemes} $\categ{Aff}_{\SS}$ will be the opposite $\infty$-category $(\Pr^{L}_{\st})^{\op}$ to the $\infty$-category of stable, presentable $\infty$-categories. For any stable presentable $\infty$-category $\AA$, we write $\Spec\AA$ for the corresponding object of $\categ{Aff}_{\SS}$.

For any object $\Spec\AA\in\categ{Aff}_{\SS}$, let us write $\Mod^l(\Spec\AA)$ for the $\infty$-category $\Fun^L(\AA,\Sp)$; we shall call these objects \emph{left modules over $\AA$}.

Since $\Sp$ is the unit in $\Pr^{L}_{\st}$, we will write
\[\Spec\CC\simeq\Spec\AA\times\Spec\BB\]
if and only if $\CC\simeq\AA\otimes\BB$ in $\Pr^L_{\st}$.
\end{dfn}

\begin{exm} By \cite[Th. 4.8.4.1]{HA}, we may note that for any $E_1$-ring $A$, one has
\[\Mod^l(\Spec(\Mod^r(A)))\simeq\Mod^l(A)\]
\end{exm}

Accordingly, then, we will prove the following, which is a generalization of Kaledin's conjecture.
\begin{thm} Suppose $N\in\NNhat$. Then for any stable, presentable $\infty$-category $\AA$, there exists an $\infty$-category $\AA_{\Psi_N}$ such that
\[\Mod(\Spec\AA_{\Psi_N})\simeq\Mod(\Spec\AA)_{\Phi_N},\]
and moreover that
\[\Spec\AA_{\Psi_N}\simeq\Spec\AA\times\Spec\Sp_{\Psi_N}.\]
\begin{proof} In light of Pr. \ref{prp:cyclonicistensor} and our characterization of cyclotomic structures, we are now reduced to studying the action of $\NN_{N_d}$ on $\AA_{\cop_N}$ via the geometric fixed points. Of course one has
\[\AA_{\Phi_N}\simeq(\AA_{\cop_N})^{h\NN_{N_d}}\simeq\Fun^L(\Sp_{\cop_N},\AA)^{h\NN_{N_d}}\simeq\Fun^L((\Sp_{\cop_N})_{h\NN_{N_d}},\AA),\]
where the colimit is formed in $\Pr^L$ over the various functors $i_n^{\star}$. This is the same as the limit in $\Pr^R$ over the right adjoint functors $i_{n,\star}$, which by \cite[Th. 5.5.3.18]{HTT} is in turn the $\infty$-categorical limit.

\begin{dfn*} The $\infty$-category of \emph{$N$-cocyclotomic objects} of $\AA$ is the $\infty$-category
\[\AA_{\Psi_N}\coloneq((\AA^{\op})_{\Phi_N})^{\op}.\]
\end{dfn*}
\noindent With this, we thus obtain the desired equivalences
\[\AA_{\Phi_N}\simeq\Fun^L(\Sp_{\Psi_N},\AA)\text{\quad and\quad}\Fun^L(\AA,\Sp)_{\Phi_N}\simeq\Fun^L(\AA_{\Psi_N},\Sp).\]
Furthermore, since $\otimes$ preserves colimits in $\Pr^L_{\st}$ separately in each variable:
\begin{eqnarray*}
\AA_{\Psi_N}&\simeq&(\AA_{\cop_N})_{h\NN_{N_d}}\\
&\simeq&((\AA\otimes\Sp)_{\cop_N})_{h\NN_{N_d}}\\
&\simeq&(\AA\otimes\Sp_{\cop_N})_{h\NN_{N_d}}\\
&\simeq&\AA\otimes(\Sp_{\cop_N})_{h\NN_{N_d}}\\
&\simeq&\AA\otimes\Sp_{\Psi_N}.
\end{eqnarray*}
This completes the proof of the Theorem.
\end{proof}
\end{thm}


\bibliographystyle{plain}
\bibliography{cyclonic}

\end{document}